\documentclass[12pt,final]{amsart}

\usepackage{enumerate}
\usepackage[OT1]{fontenc}
\usepackage[utf8]{inputenc}
\usepackage{amssymb,amsfonts,amsmath,amsthm}
\usepackage{mathtools}
\usepackage{tikz}
\usepackage{tikz-cd}
\usepackage{mathrsfs}
\usepackage{microtype}
\usepackage[hidelinks]{hyperref}
\tikzset{node distance=2cm,auto}

\usepackage{cleveref}

\theoremstyle{remark}
\newtheorem{example}{Example}[section]

\theoremstyle{definition}
\newtheorem{definition}[example]{Definition}
\theoremstyle{plain}
\newtheorem{proposition}[example]{Proposition}
\newtheorem{corollary}[example]{Corollary}

\newtheorem{theorem}[example]{Theorem}
\newtheorem{lemma}[example]{Lemma}

\newcommand{\bA}{\mathbf{A}}
\newcommand{\bC}{\mathbf{C}}
\newcommand{\bQ}{\mathbf{Q}}
\newcommand{\bR}{\mathbf{R}}
\newcommand{\bZ}{\mathbf{Z}}
\newcommand{\fin}{\mathrm{fin}}

\newcommand{\scX}{\mathscr{X}}
\newcommand{\scY}{\mathscr{Y}}
\newcommand{\scZ}{\mathscr{Z}}

\newcommand{\scA}{\mathscr{A}}
\newcommand{\scP}{\mathscr{P}}
\newcommand{\scG}{\mathscr{G}}

\newcommand{\cX}{\mathcal{X}}
\newcommand{\cY}{\mathcal{Y}}
\newcommand{\cZ}{\mathcal{Z}}

\newcommand{\cO}{\mathcal{O}}
\newcommand{\widehatO}[1]{\widehat{\cO}_k^{\,#1}}

\newcommand{\GL}{\operatorname{GL}}
\newcommand{\SL}{\operatorname{SL}}
\newcommand{\PGL}{\operatorname{PGL}}
\newcommand{\Spec}{\operatorname{Spec}}

\newcommand{\Isom}{\operatorname{Isom}}
\newcommand{\Hom}{\operatorname{Hom}}
\newcommand{\Aut}{\operatorname{Aut}}
\newcommand{\Gal}{\operatorname{Gal}}
\newcommand{\Ext}{\operatorname{Ext}}
\newcommand{\Res}{\operatorname{Res}}

\newcommand{\Br}{\operatorname{Br}}
\newcommand{\Brun}{\operatorname{Br}_{\mathrm{un}}}
\newcommand{\Bre}{\operatorname{Br}_{e}}
\newcommand{\Pic}{\operatorname{Pic}}
\newcommand{\inv}{\operatorname{inv}}
\newcommand{\Ann}{\operatorname{Ann}}
\newcommand{\ev}{\operatorname{ev}}
\newcommand{\Gm}{\mathbb{G}_m}
\newcommand{\Ga}{\mathbb{G}_a}
\DeclareFontFamily{U}{wncy}{}
\DeclareFontShape{U}{wncy}{m}{n}{<->wncyr10}{}
\DeclareSymbolFont{cyrletters}{U}{wncy}{m}{n}
\DeclareMathSymbol{\Sha}{\mathord}{cyrletters}{"58}
\newcommand{\lcm}{\operatorname{lcm}}

\title{Adelic Points and Unramified Brauer Approximation for Classifying Stacks}
\author{Ajneet Dhillon}
\email{adhill3@uwo.ca}
\hypersetup{
  pdftitle={Adelic Points and Unramified Brauer Approximation for Classifying Stacks},
  pdfauthor={Ajneet Dhillon},
  pdfsubject={Adelic points, classifying stacks, Brauer--Manin obstruction, and strong approximation}
}

\begin{document}

\begin{abstract}
Let $k$ be a number field. We study adelic topologies and strong approximation for classifying stacks.
For a connected linear algebraic group $G$, we relate the Brauer--Manin obstruction on $BG$ to Kottwitz's localisation sequence and compute
\[
\Brun(BG)/\Br(k)\simeq \Sha^1_{\mathrm{cyc}}(k,\widehat G).
\]
This yields a local criterion for strong approximation with respect to the unramified Brauer group.
Finally, for an abelian variety $A$, we compute the cohomological Brauer group of $BA$ and show that, when $A^\vee(k)$ has positive rank,
 strong approximation with respect to the full Brauer group fails off every finite set of places.
\end{abstract}

\maketitle

\section{Introduction}

Let $G$ be a connected linear algebraic group over a number field $k$.
Denote by $BG$ the classifying stack of $G$-torsors, so that
$|BG(K)|=H^1(K,G)$ for every field extension $K/k$.  Building on the
full-Brauer strong approximation theorem of
\cite[Theorem~5.5]{DhillonClassifying}, we compare the Brauer--Manin
obstruction with the Kottwitz localisation sequence
\cite[Corollary~2.5 and Proposition~2.6]{Kottwitz} and study approximation
with respect to the unramified Brauer group.

We first establish the adelic framework needed for the comparison with
Kottwitz's theorem.  If $\cX$ is an adelically liftable Artin stack
(\cref{def:adelically-liftable}) with integral model
$\scX/\cO_{k,S}$, then \cref{thm:stack-product-formula} gives a natural
equivalence
\[
 \cX(\bA_k)\simeq
 \varinjlim_{T\supseteq S\cup\Omega_k^\infty}
 \left(
   \prod_{v\in T}\cX(k_v)\times
   \prod_{v\notin T}\scX(\cO_v)
 \right).
\]
Quotient stacks by linear algebraic groups provide a broad class of examples
(\cref{prop:quotient-stacks-liftable}).  When a topological adelic lifting
presentation (\cref{def:topological-lifting-presentation}) exists, the product
topology agrees with the presentation topology, and arbitrary morphisms induce
continuous maps on adelic points
(\cref{thm:comparison-topologies,thm:adelic-functoriality}).

Let $e:\Spec k\to BG$ classify the trivial torsor, and put
\[
 \Bre(BG):=\ker\bigl(e^*:\Br(BG)\longrightarrow\Br(k)\bigr).
\]
Write
\[
 \lambda_G:\prod\nolimits'_{v\in\Omega_k}H^1(k_v,G)
 \longrightarrow\Pic(G)^D
\]
for the Kottwitz map, reviewed in \cref{subsec:kottwitz-comparison}.
 Kottwitz constructs an exact sequence of
pointed sets
\[
 H^1(k,G)
 \xrightarrow{\operatorname{loc}}
 \prod\nolimits'_{v\in\Omega_k}H^1(k_v,G)
 \xrightarrow{\lambda_G}
 \Pic(G)^D.
\]
See \cite[Theorem~3.1]{ColliotTheleneXu}; for Kottwitz's original dual-group
formulation, see
\cite[Corollary~2.5 and Proposition~2.6]{Kottwitz}.

\begin{theorem}
\label{thm:intro-BG-kottwitz}
Let $G$ be a connected linear algebraic group over a number field $k$.
For every finite set $S\subseteq\Omega_k$, there is an identification
\[
 |BG(\bA_k^S)|=\prod\nolimits_{v\notin S}'H^1(k_v,G),
\]
and this space is discrete.  Further
\[
 \Br(BG)=\Br(k)\oplus\Bre(BG),
 \qquad
 \Bre(BG)\simeq\Br(BG)/\Br(k)\simeq\Pic(G).
\]
If $L\in\Pic(G)$ and $b_L\in\Bre(BG)$ is the corresponding
Brauer class, then the Kottwitz map satisfies
\[
 \lambda_G((\xi_v)_v)(L)
 =\sum_{v\in\Omega_k}\inv_v\bigl(b_L(\xi_v)\bigr)
\]
and hence
\[
 BG(\bA_k)^{\Br(BG)}
 =\ker(\lambda_G)
 =\operatorname{im}\bigl(BG(k)\longrightarrow BG(\bA_k)\bigr).
\]
\end{theorem}

The discreteness assertion is proved in
\cref{prop:local-BG-finite-discrete,thm:BG-adelic-discrete}.  The Brauer-group
identification and the comparison of the two obstruction maps are proved in
\cref{prop:Pic-Br-BG,thm:kottwitz-global-sequence}; the final 
equality is \cref{cor:kottwitz-Brauer-localization}.

We next study approximation with respect to the unramified Brauer group. We start by proving
a formal modification theorem
(\cref{prop:abstract-evaluation-formal,thm:abstract-formal-modification}) that is
analogous to \cite[Corollaire~2.6.1]{Harari}.
For certain quotient stacks, a residue formula for unramified Brauer groups
is proved in \cref{thm:quotient-stack-unramified-residue}. This can be applied to
classifying stacks to obtain a description of the unramified Brauer group,
\cref{thm:Brun-BG-Sha}.
This finally leads to an approximation criterion for $BG$.
To state the result, put $\widehat G=X^*(G_{\bar{k}})$.  Via
$\Bre(BG)\simeq\Pic(G)$, let $U_G\subseteq\Pic(G)$ correspond to
$\Brun(BG)\cap\Bre(BG)$.  For every place $v$, let
\[
 \kappa_v:H^1(k_v,G)\longrightarrow\Pic(G)^D
\]
be the local Kottwitz map, put $E_v=\operatorname{im}(\kappa_v)$ and
$E_S=\sum_{v\in S}E_v$, and let $\Ann(U_G)\subseteq\Pic(G)^D$ be the
annihilator of $U_G$.  We use $\Sha^1_{\mathrm{cyc}}(k,\widehat G)$ for the
cyclically unramified cohomology group defined in
\cref{def:cyclically-unramified}.

\begin{theorem}[The unramified Brauer group and approximation]
\label{thm:intro-Brun-BG}
Let $G$ be a connected linear algebraic group over $k$.
There is a natural isomorphism
\[
 \Brun(BG)/\Br(k)
 \simeq \Sha^1_{\mathrm{cyc}}(k,\widehat G)\simeq U_G.
\]
For every finite set $S\subseteq\Omega_k$, $BG$ satisfies strong approximation
off $S$ with respect to $\Brun(BG)$ if and only if
\[
 \Ann(U_G)\subseteq E_S.
\]
\end{theorem}

The computation of $\Brun(BG)$ is \cref{thm:Brun-BG-Sha}.  The description
of the Brauer--Manin set and the approximation criterion are
\cref{prop:projected-unramified-set,thm:S-local-criterion}; the equivalence
between density and surjectivity follows from \cref{thm:BG-adelic-discrete}.

The examples in \cref{sec:examples} make this criterion explicit.  For
$B\PGL_n$ it becomes a condition on the sum of the local invariants of central simple
algebras; for a cyclic norm-one torus it is equivalent to the decomposition
groups at the places in $S$ generating the Galois group.  A biquadratic torus
shows that unramified Brauer approximation may hold even when classical
strong approximation fails
(\cref{thm:PGLn-approximation-criterion,thm:cyclic-norm-one-criterion,thm:biquadratic-Sha-example}).

Finally, \cref{sec:abelian-varieties} treats the classifying stack $BA$ of an abelian
variety.
Although $BA$ is adelically liftable, it has no topological adelic lifting
presentation when $A\neq0$; nevertheless,
$|BA(\bA_k^S)|=\prod_{v\notin S}'H^1(k_v,A)$ with the discrete product 
topology
(\cref{prop:relative-classifying-liftable,prop:BA-not-topologically-liftable,prop:BA-adelic-discrete}).
Here we use the convention that the cohomology of $BA$ is computed on its
big \'etale site, and
$\Br(BA)=H^2_{\mathrm{\acute et}}(BA,\Gm)_{\mathrm{tors}}$.

\begin{theorem}[The Brauer group and approximation for $BA$]
\label{thm:intro-BA}
Let $A$ be an abelian variety over $k$.  There are natural split
decompositions
\[
 H^2_{\mathrm{\acute et}}(BA,\Gm)\simeq\Br(k)\oplus A^\vee(k),
 \qquad
 \Br(BA)\simeq\Br(k)\oplus A^\vee(k)_{\mathrm{tors}}.
\]
If $A^\vee(k)$ has positive rank, then for every finite set
$S\subseteq\Omega_k$ one has
\[
 \operatorname{im}\bigl(BA(k)\longrightarrow BA(\bA_k^S)\bigr)
 \subsetneq BA(\bA_k^S)^{\Br(BA)}.
\]
Consequently, $BA$ fails strong approximation off $S$ with respect to its
full cohomological Brauer group.
\end{theorem}

The Brauer-group computation is \cref{prop:brauer}, and the strict inclusion
and failure of strong approximation are proved in
\cref{thm:BA-full-Brauer-failure}.

\section*{Notation}
\label{sec:notation}

\begin{tabular}{p{0.18\textwidth}p{0.74\textwidth}}
    $k$  & A number field \\
    $\bA_k$ & The adele ring of $k$\\
    $\bA_k^S$ & The adeles away from a finite set of places $S$, equivalently the image of $\bA_k$ under projection to the factors indexed by $v\notin S$\\
    $\Omega_k$ & The set of places of $k$ \\
    $\Omega_k^\infty$ & The set of infinite places of $k$ \\
    $\Omega_k^\fin$ & The set of finite places of $k$ \\
    $\cO_v$ & The valuation ring for $v\in\Omega_k^\fin$ \\
    $\cO_{k,S}$ & The ring of $S$-integers. Here $S\subseteq \Omega_k$ is a finite set. We define \\
    & $\{x\in k \mid x\in\cO_v\ \forall\ v\notin S\}$ \\
    $k_T$  & For a finite set $T\subseteq \Omega_k$ this is defined to be the product $\prod_{v\in T}k_v$.\\
    $\widehatO{T}$ & For a finite set $T\supseteq \Omega_k^\infty$, the integral adelic ring $\prod_{v\notin T}\cO_v\subseteq \bA_k^T$ \\
    $\bA_k(T)$ & For a finite set $T\supseteq \Omega_k^\infty$, the ring $\prod_{v\in T}k_v\times \widehatO{T}$ \\
    $|\cX(R)|$ & The set $\pi_0(\cX(R))$ of isomorphism classes of $R$-points \\
    $X\times^G H$ & If $G$ acts on the left on $X$ and $G\hookrightarrow H$, the contracted product $(X\times H)/G$, where $(x,h)\cdot g=(g^{-1}x,hg)$; it carries the residual left $H$-action. \\
\end{tabular}

For an algebraic stack $\cZ$ over a field, we write
$H^i_{\mathrm{\acute et}}(\cZ,-)$ for cohomology on its big \'etale site and
use the cohomological Brauer group
\[
 \Br(\cZ):=H^2_{\mathrm{\acute et}}(\cZ,\Gm)_{\mathrm{tors}}.
\]

For an abelian group $M$, write
\[
 M^D:=\Hom(M,\bQ/\bZ).
\]

We use ordinary capital letters such as $X,Y,Z$ for algebraic spaces, schemes, or varieties.

Algebraic stacks over $k$ are denoted by calligraphic letters such as $\cX,\cY,\cZ$.

Integral models of stacks, spaces, schemes, and varieties are denoted by script letters such as $\scX,\scY,\scZ$.

Whenever groupoid structure matters, $\cX(R)$ denotes the groupoid of
$R$-points and $|\cX(R)|$ its set of isomorphism classes.  In statements about
approximation, topology, or Brauer--Manin sets, including those in the
Introduction and Sections~3--5, spaces mean sets of isomorphism classes
and we suppress the vertical bars.

\noindent\textbf{Acknowledgements.}
I thank Azur {\DJ}onlagi\'c for observing that the hypothesis on the omitted
places in \cite[Theorem~5.5]{DhillonClassifying} can be substantially weakened
using the homogeneous-space theorem of Borovoi--Demarche, and for directing me
to \cite{Cesnavicius}; Daniel Loughran for suggesting the study of the
unramified Brauer group of algebraic stacks; and Patrick Brosnan for pointing
out the relevance of R.~Kottwitz's work to this problem.

\section*{AI Statement}

Artificial intelligence tools were used for research and the exploration of examples.
After rewriting the initial version, AI tools were used to check for errors and correct
typesetting, grammar and other mistakes.

\section{The product formula for stacks}

We begin by stating a variation of a result
from \cite{Conrad} that will be needed in the proof of \cref{thm:stack-product-formula} below.

\begin{lemma}\label{lem:conrad-extn}
    Fix a finite set $T\supseteq S\cup\Omega_k^\infty$, put $R=\widehatO{T}$, and let
$W$ be a separated algebraic space of finite presentation over $R$.  Then the
natural map
\[
 W(R)\longrightarrow
 \prod_{v\notin T}W(\cO_v)
\]
is a bijection.
\end{lemma}

\begin{proof}
    This follows as in the proof of
    \cite[Proposition 5.2]{Conrad}.
\end{proof}

We shall use the following from \cite{Conrad}.

\begin{proposition}\label{prop:conrad-product}
Let $\scZ$ be a separated algebraic space of finite presentation over
$\cO_{k,S}$, with generic fibre $Z$.  Then the natural map
\[
 Z(\bA_k)\longrightarrow
 \varinjlim_{T\supseteq S\cup\Omega_k^\infty}
 \left(
   \prod_{v\in T} Z(k_v)\times
   \prod_{v\notin T}\scZ(\cO_v)
 \right)
\]
is a bijection.
\end{proposition}

\begin{proof}
See
\cite[Theorems~3.4 and~3.6, Proposition~5.2]{Conrad}.
\end{proof}

\begin{definition}\label{def:adelically-liftable}
Let $\scX$ be an Artin stack of finite presentation over $\cO_{k,S}$.  We say
that $\scX$ is \emph{adelically liftable} if there is a smooth surjective presentation 
$p:U\to\scX$, defined over $\cO_{k,S}$, such that the following hold:
\begin{enumerate}[(i)]
\item the diagonal
\[
 \Delta_{\scX/\cO_{k,S}}:\scX\longrightarrow
 \scX\times_{\cO_{k,S}}\scX
\]
is representable by separated algebraic spaces of finite presentation;
\item the algebraic space $U$ is separated and of finite presentation over
$\cO_{k,S}$;
\item for every finite place $v\notin S$, the functor
\[
 U(\cO_v)\longrightarrow \scX(\cO_v)
\]
is essentially surjective;
\end{enumerate}
Such a presentation $p$ will be called an \emph{adelic lifting presentation}.
We say that a finite type Artin stack over $k$ is adelically liftable if it
admits an adelically liftable integral model for some finite $S$.
\end{definition}

\begin{definition}\label{def:topological-lifting-presentation}
Let $\scX/\cO_{k,S}$ be as in \cref{def:adelically-liftable}, with generic
fibre $\cX$.  An adelic lifting presentation $p:U\to\scX$ is called a
\emph{topological adelic lifting presentation} if, in addition, the following
conditions hold:
\begin{enumerate}[(i)]
\item for every place $v$, the map $|U(k_v)|\longrightarrow |\cX(k_v)|$ is surjective;
\item for every finite place $v\notin S$, the map $|U(\cO_v)|\longrightarrow |\scX(\cO_v)|$ is surjective;
\item for every finite set $T\supseteq S\cup\Omega_k^\infty$, the map $|U(\bA_k(T))|\longrightarrow |\scX(\bA_k(T))|$ is surjective.
\end{enumerate}
We say that $\cX$ is
\emph{topologically adelically liftable} if some integral model of it has a
topological adelic lifting presentation.
\end{definition}

In \cref{sec:abelian-varieties}, we will show that the stack $BA$, where $A$ is an abelian variety, is adelically liftable
but not topologically adelically liftable.

The following proposition supplies numerous examples of stacks that are topologically adelically liftable.

\begin{proposition}\label{prop:quotient-stacks-liftable}
Let $\scZ$ be a separated algebraic space of finite presentation over
$\cO_{k,S}$, and let $\scG$ be a flat affine group scheme of finite presentation
acting on $\scZ$.  Suppose that there is a closed
immersion
\[
 \scG\hookrightarrow \GL_{n,\cO_{k,S}}.
\]
Then the quotient stack $[\scZ/\scG]$ is topologically adelically liftable, with a topological adelic lifting presentation given by

\[
 \scZ\times^{\scG}\GL_n\to  [\scZ\times^{\scG}\GL_n/\GL_n]\simeq [\scZ/\scG].
\]
\end{proposition}

\begin{proof}
Set $\scY:=\scZ\times^{\scG}\GL_n$.
We check the criteria of \cref{def:adelically-liftable,def:topological-lifting-presentation}.
\begin{enumerate}
  [(i)]
  \item By \cite[Tags~0DTY and~04XD]{StacksProject}, separatedness and finite
presentation of the diagonal of $[\scZ/\scG]$ may be checked on the
relation morphism
\[
 (s,t): \scG\times_{\cO_{k,S}} \scZ \longrightarrow
 \scZ\times_{\cO_{k,S}}\scZ,
 \qquad
 (g,z)\longmapsto (z,gz).
\]
Writing $a:\scG\times\scZ\to\scZ$ for the action, this morphism
factors as
\[
 \scG\times\scZ
 \xrightarrow{\Gamma_a}
 \scG\times\scZ\times\scZ
 \longrightarrow
 \scZ\times\scZ.
\]
Since $\scZ$ is separated and of finite presentation,
$\Gamma_a$ is a closed immersion of finite presentation
\cite[Tags~03KO and~084P]{StacksProject}.  The second arrow is the
base change of the affine morphism of finite presentation
$\scG\to\Spec\cO_{k,S}$.  Hence $(s,t)$ is affine and of finite
presentation, and is therefore separated and of finite presentation.
The same consequently holds for the diagonal of $[\scZ/\scG]$.
  \item This is a minor variation on the argument in \cite[\S2.4 after Definition 2.14]{DhillonClassifying}.
  \item This follows from the fact that every projective module over a local ring is free and the description
  of points on a quotient stack.
\end{enumerate}
To see that the prescribed presentation is a topological adelic lifting presentation, the reader is referred to \cite[\S 2.4]{DhillonClassifying}.
\end{proof}

\begin{theorem}\label{thm:stack-product-formula}
Let $\cX$ be a finite type Artin stack over $k$, and let
$\scX/\cO_{k,S}$ be an adelically liftable integral model with generic fibre
$\cX$.  Then there is a natural equivalence of groupoids
\[
 \cX(\bA_k)\simeq
 \varinjlim_{T\supseteq S\cup\Omega_k^\infty}
 \left(
   \prod_{v\in T}\cX(k_v)\times
   \prod_{v\notin T}\scX(\cO_v)
 \right).
\]
\end{theorem}

\begin{proof}
Fix a finite set $T\supseteq S\cup\Omega_k^\infty$.  We first show that the restriction functor
\[
 \scX(\widehatO{T})\longrightarrow \prod_{v\notin T}\scX(\cO_v)
\]
is an equivalence of groupoids.  For essential surjectivity, let
$(x_v)_{v\notin T}$ be a family of objects with $x_v\in \scX(\cO_v)$.  Choose
lifts $u_v\in U(\cO_v)$ along the adelic lifting presentation.  By
Lemma~\ref{lem:conrad-extn}, applied to the separated algebraic space
$U$, the family $(u_v)_{v\notin T}$ is induced by a unique point
$u\in U(\widehatO{T})$.  Then $p(u)$ restricts to an object isomorphic to $x_v$ for each
$v\notin T$.

For full faithfulness, let $x,y\in \scX(\widehatO{T})$ and set
\[
 I:=\Isom_{\scX_{\widehatO{T}}}(x,y).
\]
By the diagonal hypothesis, $I$ is a separated algebraic space of finite
presentation over $\widehatO{T}$.  Lemma~\ref{lem:conrad-extn} gives
\[
 I(\widehatO{T})\cong \prod_{v\notin T} I(\cO_v),
\]
which is exactly the required bijection on isomorphism sets.

Since
\[
 \bA_k(T)=\prod_{v\in T}k_v\times \widehatO{T},
\]
and a stack sends finite disjoint unions of affine schemes to products of
groupoids, we obtain
\[
 \scX(\bA_k(T))
 \simeq
 \prod_{v\in T}\cX(k_v)\times
 \prod_{v\notin T}\scX(\cO_v).
\]
Finally, $\bA_k=\varinjlim_{T\supseteq S\cup\Omega_k^\infty}\bA_k(T)$, and $\scX$ is locally of finite
presentation, hence limit preserving on affine schemes \cite[Tag 0CMQ]{StacksProject}.  Taking the filtered
colimit over $T$ gives the claimed equivalence.
\end{proof}


\subsection{Topologies}

  Let $R$ be a local topological ring
with continuous inversion.
For an algebraic stack $\cZ$ locally of finite type over $R$, following
\cite[Section~2.4]{Cesnavicius}, a full subcategory
$\mathcal U\subseteq \cZ(R)$ stable under isomorphism is declared open if, for
every $R$-morphism $f:Z\to \cZ$ from a locally of finite type $R$-scheme, the
inverse image $f(R)^{-1}(\mathcal U)$ is open in $Z(R)$.  This gives a topology
on $|\cZ(R)|$.  We apply this construction to $R=k_v$ and, for finite $v$, to
$R=\cO_v$.  For algebraic spaces it agrees with the usual topology on
$R$-points, and smooth morphisms induce open maps for these local rings
\cite[Section~2.8 and Proposition~2.9(a)]{Cesnavicius}.

There are two ways to topologise the adelic points of an adelically 
liftable algebraic stack. We will now describe these and show
that they produce equivalent topologies.

Let $\scX/\cO_{k,S}$ be an adelically liftable algebraic stack, with
generic fibre $\cX$ and topological adelic lifting presentation $p:U\to\scX$.
The \emph{presentation topology} on $|\cX(\bA_k)|$ is defined by giving
each $|\scX(\bA_k(T))|$ the quotient topology induced from $|U(\bA_k(T))|$ and then
taking the colimit over $T$.
The \emph{product topology} on $|\cX(\bA_k)|$ is defined by endowing
\[
 \prod_{v\in T}|\cX(k_v)|\times
 \prod_{v\notin T}|\scX(\cO_v)|
\]
with the product topology, for each finite set $T\supseteq S\cup\Omega_k^\infty$ and then taking the colimit topology over all such $T$.

\begin{theorem}\label{thm:comparison-topologies}
These two topologies agree.
\end{theorem}

\begin{proof}
   Let $R$ be either $k_v$ or, for finite
$v\notin S$, $\cO_v$.  We claim that if the map on isomorphism classes induced by the
presentation is surjective, then
\[
 |U(R)|\longrightarrow |\scX(R)|
\]
(with $\scX_{k_v}=\cX_{k_v}$ when $R=k_v$) is an open quotient map for the local
\v{C}esnavi\v{c}ius topologies.  Indeed, it is continuous by functoriality of the
stack topology.  It is open because $p$ is smooth and the rings $k_v$ and
$\cO_v$ satisfy the hypotheses under which smooth morphisms of algebraic stacks
are open on $R$-points \cite[Section~2.8 and Proposition~2.9(a)]{Cesnavicius}.
Since the map is surjective, openness implies that it is a quotient map.

Now fix a finite set $T\supseteq S\cup\Omega_k^\infty$.  By the topology just specified for separated
algebraic spaces over $\bA_k(T)$, \cite[Propositions~5.2 and~5.4, and the discussion following Corollary~5.7]{Conrad}
identifies
$|U(\bA_k(T))|$ homeomorphically with
\[
 U_T:=
 \prod_{v\in T}|U(k_v)|\times
 \prod_{v\notin T}|U(\cO_v)| .
\]
Set
\[
 X_T:= \prod_{v\in T}|\cX(k_v)|\times
 \prod_{v\notin T}|\scX(\cO_v)|.
\]

Under this identification we have a commutative square
\[
\begin{tikzcd}
 \pi_0(U(\bA_k(T))) \arrow[r,"\sim"] \arrow[d] & U_T \arrow[d] \\
 \pi_0(\scX(\bA_k(T))) \arrow[r,"\Phi_T"] & X_T.
\end{tikzcd}
\]
The left vertical map is the quotient map defining the presentation-quotient
topology on $|\scX(\bA_k(T))|$.  The right vertical map is the product of the
local maps
\[
 |U(k_v)|\twoheadrightarrow |\cX(k_v)|,
 \qquad
 |U(\cO_v)|\twoheadrightarrow |\scX(\cO_v)| .
\]
Each factor is an open quotient map.  Their product
is again an open quotient map.

It follows that a subset $W\subseteq X_T$ is open if and only if its inverse image
in $U_T$ is open.  Via the top horizontal homeomorphism, this is equivalent to
openness in $|U(\bA_k(T))|$ of the inverse image of $\Phi_T^{-1}(W)$, which by
the left quotient topology is equivalent to openness of $\Phi_T^{-1}(W)$ in
$|\scX(\bA_k(T))|$.  Thus $\Phi_T$ is a homeomorphism.  These homeomorphisms
are compatible with enlarging $T$, so they induce a homeomorphism on colimits.
\end{proof}

In \cite[Proposition~2.16]{DhillonClassifying}, a functoriality theorem was proved for representable morphisms. The new product description of the
topology allows us to prove a stronger functoriality result.

\begin{theorem}\label{thm:adelic-functoriality}
Let $\scX$ and $\scY$ be adelically liftable Artin stacks over $\cO_{k,S}$ with generic fibres
$\cX$ and $\cY$.  Let
\[
 f:\scX\longrightarrow \scY
\]
be any morphism over $\cO_{k,S}$.  Then the induced map
\[
 |f(\bA_k)|:|\cX(\bA_k)|\longrightarrow |\cY(\bA_k)|
\]
is continuous.
\end{theorem}

\begin{proof}
  Let $R=k_v$, or let $R=\cO_v$ for a finite place $v\notin S$, and write
$\scX_R$ and $\scY_R$ for the relevant base changes.  The topology on
isomorphism classes of $R$-points is functorial directly from its
definition.

The theorem now follows from the fact that the induced map on product topologies is continuous
and compatibility with colimits.

\end{proof}


\subsection[The topology of BG]{The topology of \texorpdfstring{$BG$}{BG}}

In this subsection $G/k$ is a connected linear algebraic group. Note that $BG$ is adelically liftable by \cref{prop:quotient-stacks-liftable}.
As a result, we have an identification
\[
|BG(\bA_k)| = \prod' |BG(k_v)|=\prod' H^1(k_v, G)
\]
that we will make use of extensively in this subsection.

\begin{theorem}
For every place $v$ of $k$, the set $H^1(k_v,G)$ is finite.
\end{theorem}

\begin{proof}
For $G$ reductive and nonarchimedean $v$, Kottwitz constructs a canonical bijection
\[
H^1(k_v,G)\xrightarrow{\sim}
A(G):=\pi_0\!\left(Z(G^\vee_{\mathrm{Lang}})^{\Gamma_v}\right)^D,
\]
where $G^\vee_{\mathrm{Lang}}$ is the Langlands dual group,
$\Gamma_v=\Gal(\overline{k_v}/k_v)$, and $D$ denotes the duality introduced in
the notation section. The group $A(G)$ is finite; see
\cite[Theorem~1.2]{Kottwitz}. For real places this follows from
\cite[Lemma~2.1 and Theorem~3.1]{BorovoiReal}. For a general connected $G$,
one passes to $G^{\mathrm{red}}:=G/R_u(G)$ and notes that the unipotent radical has a
filtration whose associated graded pieces are $\Ga$.
\end{proof}

\begin{lemma}\label{lem:local-torsor-classes-discrete}
Let $F$ be a local field of characteristic zero, let $G$ be a smooth algebraic
group over $F$, let $Y$ be a separated finite type $F$-algebraic space, and let
$P\to Y$ be a $G$-torsor.  The class map
\[
 Y(F)\longrightarrow H^1(F,G),
 \qquad y\longmapsto[P_y],
\]
is locally constant for the topology on stacky $F$-points.
\end{lemma}

\begin{proof}
Fix $y_0\in Y(F)$ and let $P_0$ be the fibre of $P$ at $y_0$.  The
isomorphism space
\[
 Q:=\Isom_Y(P_0\times_F Y,P)
\]
is a torsor over $Y$ under the smooth inner form $\Aut_G(P_0)$ of $G$, and
its fibre over $y_0$ has an $F$-point.  Smooth morphisms of finite type
algebraic spaces over a local field are open on $F$-points
\cite[Proposition~2.9(a)]{Cesnavicius}.  Thus the image of $Q(F)$ contains an
open neighbourhood of $y_0$, on which every fibre of $P$ is isomorphic to
$P_0$.  Since $y_0$ was arbitrary, the class map is locally constant.
\end{proof}

\begin{proposition}\label{prop:local-BG-finite-discrete}
Let $G$ be a connected linear algebraic group over $k$.  For every place $v$ of
$k$, the topological space
\[
 |BG(k_v)|=H^1(k_v,G)
\]
is finite and discrete.
\end{proposition}

\begin{proof}
Finiteness follows from the preceding theorem.  Choose a closed immersion
$G\hookrightarrow\GL_n$.  Since $H^1(k_v,\GL_n)=1$, the smooth presentation
$\GL_n/G\to BG$ is surjective on $k_v$-points.  Its class map is locally
constant by \cref{lem:local-torsor-classes-discrete} and is a quotient map for
the stack topology.  Hence every singleton in $H^1(k_v,G)$ is open.
\end{proof}

\begin{lemma}\label{lem:integral-BG-triviality}
Let $G$ be a connected linear algebraic group over $k$.  After enlarging $S$,
one may choose an integral model $\scG/\cO_{k,S}$ with connected geometric
fibres, and then
\[
 H^1(\cO_v,\scG)=1
\]
for every finite place $v\notin S$.
\end{lemma}

\begin{proof}
Smoothness and connectedness hold after enlarging $S$. If
$\scP\to\Spec\cO_v$ is a torsor, then its special fibre is a torsor under
a connected algebraic group over the residue field and hence has a rational
point by Lang's theorem \cite[Theorem~2]{Lang}. This point lifts to a section
over $\Spec\cO_v$ by smoothness and Hensel's lemma.
\end{proof}

\begin{theorem}
\label{thm:BG-adelic-discrete}
Let $G$ be a connected linear algebraic group over $k$, and let
$S\subseteq\Omega_k$ be any finite set of places.
The topology on $|BG(\bA_k^S)|$ is discrete.
\end{theorem}

\begin{proof}
  The topology on $|BG(\bA_k(T))|$ is discrete by the above results and the definition of the product topology.
  The result follows from the fact that a filtered colimit of discrete spaces is discrete.
\end{proof}

\section{Strong Approximation}

For $v\in\Omega_k$, let
\[
 \inv_v:\Br(k_v)\longrightarrow\bQ/\bZ
\]
denote the local invariant map.

\begin{lemma}\label{lem:BM-finite-support}
Let $\cX$ be a finite type algebraic stack over a number field $k$.  For every
adelic point $x=(x_v)_v\in\cX(\bA_k)$ and every $b\in\Br(\cX)$, one has
\[
 \inv_v\bigl(b(x_v)\bigr)=0
\]
for all but finitely many places $v$.  Consequently, the Brauer--Manin pairing
\[
 \cX(\bA_k)\times\Br(\cX)\longrightarrow\bQ/\bZ,
 \qquad
 (x,b)\longmapsto\sum_{v\in\Omega_k}\inv_v\bigl(b(x_v)\bigr),
\]
is well-defined.
\end{lemma}

\begin{proof}
This is \cite[Proposition~4.4]{DhillonClassifying}.  See also
\cite[\S4.2, especially (4.5) and the paragraph immediately following it,
p.~15]{LvWu}.
\end{proof}

If $x\in\cX(\bA_k)$, we write $x_v$ for its pullback along
\[
 \Spec k_v\longrightarrow\Spec\bA_k.
\]
We abuse notation and write $x=(x_v)_v$.  When $\cX$ is adelically
liftable, this notation is justified by the product formula of
Theorem~\ref{thm:stack-product-formula}; for a general finite type
algebraic stack, such a product description need not hold.
For a subgroup $C\subseteq \Br(\cX)$, write
\[
 \cX(\bA_k)^C
 :=
 \left\{x=(x_v)_v\in \cX(\bA_k):
 \sum_v\inv_v b(x_v)=0\text{ for every }b\in C\right\}.
\]

\begin{definition}
\label{def:strong-approximation-off-S}
Let $S\subseteq \Omega_k$ be finite, and let
\[
 p^S:\cX(\bA_k)\longrightarrow \cX(\bA_k^S)
\]
be the projection away from $S$.  For $C\subseteq \Br(\cX)$ put
\[
 \cX(\bA_k^S)^C:=p^S\bigl(\cX(\bA_k)^C\bigr).
\]
We say that $\cX$ satisfies \emph{strong approximation off $S$ with respect to
$C$} if the diagonal image of $\cX(k)$ is dense in $\cX(\bA_k^S)^C$.

When $C=0$, this says that the diagonal image of $\cX(k)$ is dense in the
space $\cX(\bA_k^S)^0$. If this holds, we say that \emph{classical strong approximation
off $S$} holds.
\end{definition}

\begin{theorem}
\label{thm:BD-Dhillon-full-Brauer}
Let $G$ be a connected linear algebraic group over $k$, and let
$S\subseteq\Omega_k$ be a nonempty finite set of places.  Then $BG$ satisfies
strong approximation off $S$ with respect to its full Brauer group:
\[
 \overline{BG(k)}=BG(\bA_k^S)^{\Br(BG)}.
\]

\end{theorem}

\begin{proof}
When $S$ contains all archimedean places and at least one finite place, this is
\cite[Theorem~5.5]{DhillonClassifying}.  The same proof applies to every
nonempty finite set $S$, as we briefly recall.

Choose a faithful representation
\[
 G\hookrightarrow H:=\SL(V)
\]
with $\dim V\geq 2$, and put $X:=H/G$.  The quotient presentation
\[
 q:X\longrightarrow [X/H]\simeq BG
\]
induces an isomorphism
\[
 q^*:\Br(BG)\xrightarrow{\sim}\Br(X)
\]
by \cite[Proposition~5.3]{DhillonClassifying}.

The group $H$ is semisimple and simply connected and satisfies strong
approximation off every nonempty $S$; see
\cite[Chapter~7, Theorem~7.12]{PlatonovRapinchuk}.  Hence
\cite[Theorem~6.1]{BorovoiDemarche}, applied to the homogeneous space $X$ with
connected stabiliser $G$, gives
\[
 X(\bA_k)^{\Br_{x_0}(X)}
 =\overline{H(k_S)\cdot X(k)}
 \qquad\text{in }X(\bA_k).
\] 
Every adelic point of $BG$ lifts through $q$, and the Brauer--Manin pairing is
functorial.  Projecting the displayed equality away from $S$ removes the factor
$H(k_S)$; mapping the resulting rational points of $X$ through $q$ therefore
gives the required density in $BG(\bA_k^S)^{\Br(BG)}$.
\end{proof}

\begin{corollary}\label{cor:BG-full-Brauer-localization}
In the above situation, we have
\[
 \operatorname{im}\bigl(BG(k)\longrightarrow BG(\bA_k^S)\bigr)
 =BG(\bA_k^S)^{\Br(BG)}.
\]
That is, the closure is not needed.
\end{corollary}

\begin{proof}
  The closure is not needed as the topology is discrete.
\end{proof}


\subsection{Comparison with the Kottwitz sequence}
\label{subsec:kottwitz-comparison}

\begin{lemma}\label{lem:Pic-central-extensions}
Let $\Ext^c_{k\text{-}\mathrm{gp}}(G,\Gm)$ denote the Baer group of central
extensions of $k$-groups
\[
 1\longrightarrow\Gm\longrightarrow E\longrightarrow G\longrightarrow 1.
\]
Forgetting the group structure on $E$ induces a natural isomorphism
\[
 \Ext^c_{k\text{-}\mathrm{gp}}(G,\Gm)
 \xrightarrow{\sim}\Pic(G),
 \qquad [E]\longmapsto [E\to G].
\]
Every extension of the connected group $G$ by $\Gm$ is automatically central.
\end{lemma}

\begin{proof}
The first part is essentially \cite[Theorem 5.6, Corollary 5.7]{ColliotTheleneFlasque}.
For the last part, notice that the
conjugation action on the kernel factors through $G\to\Aut(\Gm)$ and is trivial
because $G$ is connected, so the extension is central.
\end{proof}

Recall that
\[
 \Pic(G)^D=\Hom(\Pic(G),\bQ/\bZ).
\]
For each place $v$, we have a restriction map
\[
 r_v:\Pic(G)\longrightarrow\Pic(G_{k_v}),
 \qquad L\longmapsto L_v.
\]
Kottwitz constructs a pairing
\[
 H^1(k_v,G)\times\Pic(G_{k_v})\longrightarrow\bQ/\bZ
\]
that sends $(\xi_v,L_v)$ to the invariant of the Brauer class obtained by applying
the boundary map for the central extension of $G_{k_v}$ corresponding to
$L_v$.  The pairing defines
\[
 \kappa_v:H^1(k_v,G)\longrightarrow\Pic(G)^D,
 \qquad
 \kappa_v(\xi_v)(L):=\langle\xi_v,L_v\rangle_v.
\]
Kottwitz's localisation sequence is
\begin{equation}\label{eq:kottwitz-global-sequence}
 H^1(k,G)
 \xrightarrow{\operatorname{loc}}
 \prod_{v\in\Omega_k}' H^1(k_v,G)
 \xrightarrow{\lambda_G}
 \Pic(G)^D,
\end{equation}
where $\operatorname{loc}$ sends a global torsor to its localisations and
\[
 \lambda_G((\xi_v)_v):=\sum_{v\in\Omega_k}\kappa_v(\xi_v).
\]

Let
\[
 e:\Spec k\longrightarrow BG
\]
be the point classifying the trivial $G$-torsor, and put
\[
 \Bre(BG):=\ker\bigl(e^*:\Br(BG)\longrightarrow\Br(k)\bigr).
\]
Choose a faithful representation $G\hookrightarrow H:=\SL(V)$, put
$X:=H/G$, and let
\[
 q:X\longrightarrow [X/H]\simeq BG,
 \qquad x_0:=1G\in X(k).
\]

\begin{proposition}\label{prop:Pic-Br-BG}
There are natural isomorphisms of finite abelian groups
\[
 \Bre(BG)\xrightarrow{\sim}\Pic(G)
\]
and
\[
 \Br(BG)/\Br(k)\xrightarrow{\sim}\Pic(G).
\]
More precisely, pullback along $q$ identifies $\Bre(BG)$ with
\[
 \Br_{x_0}(X):=
 \ker\bigl(x_0^*:\Br(X)\longrightarrow\Br(k)\bigr).
\]
\end{proposition}

\begin{proof}
The variety $X=H/G$ is smooth.  For a smooth variety equipped with an
action of $H=\SL(V)$, pullback through the quotient presentation induces an
isomorphism
\[
 q^*:\Br([X/H])\xrightarrow{\sim}\Br(X);
\]
this is \cite[Proposition~5.3]{DhillonClassifying}.  Under the equivalence
$[X/H]\simeq BG$, the point $x_0$ maps to the point $e$ classifying the
trivial $G$-torsor.  Consequently,
\[
 x_0^*\circ q^*=e^*,
\]
and $q^*$ restricts to an isomorphism
$$
 q^*:\Bre(BG)\xrightarrow{\sim}\Br_{x_0}(X).
$$

The result follows by applying \cite[Proposition~2.10(ii)]{ColliotTheleneXu} to
the homogeneous space
\[
 X=H/G,
 \qquad H=\SL(V).
\]
The ambient group $H$ is semisimple and simply connected, and the
stabiliser $G$ is connected.  The map associated with the $G$-torsor
$H\to H/G$ therefore induces an isomorphism
\[
 \Pic(G)\xrightarrow{\sim}\Br_{x_0}(X).
\]
Finally, $\Pic(G)$ is finite by \cite[Proposition~6.10]{Sansuc}.
\end{proof}

\begin{theorem}\label{thm:kottwitz-global-sequence}
The sequence \eqref{eq:kottwitz-global-sequence} is exact.  For
$L\in\Pic(G)$,
let $b_L\in\Bre(BG)$ be the Brauer class corresponding to $L$ under
Proposition~\ref{prop:Pic-Br-BG}.  Then
$$
 \lambda_G((\xi_v)_v)(L)
 =\sum_{v\in\Omega_k}\inv_v\bigl(b_L(\xi_v)\bigr).
$$
Thus $\lambda_G$ is precisely the Brauer--Manin obstruction map associated to
the normalised Brauer group of $BG$.
\end{theorem}

\begin{proof}
Exactness in this form, for connected linear algebraic groups, is
\cite[Theorem~3.1]{ColliotTheleneXu}.  For reductive groups, see
\cite[Theorem~9.4(iv)]{ColliotTheleneFlasque}.  Kottwitz's original
formulation is
\cite[Corollary~2.5 and Proposition~2.6]{Kottwitz}.
The general formula follows from Kottwitz by filtering the unipotent radical with associated graded pieces
isomorphic to $\Ga$.

It remains to compare the local maps $\kappa_v$ with the Brauer--Manin pairing.
Let
\[
 1\longrightarrow\Gm\longrightarrow E_L\longrightarrow G\longrightarrow 1
\]
be the central extension whose underlying $\Gm$-torsor represents
$L\in\Pic(G)$, as in Lemma~\ref{lem:Pic-central-extensions}.  After base
change to $k_v$, this extension gives a connecting map
\[
 \partial_{L,v}:H^1(k_v,G)\longrightarrow
 H^2(k_v,\Gm)=\Br(k_v).
\]
Proposition~2.9 of \cite{ColliotTheleneXu} identifies
$\partial_{L,v}(\xi_v)$ with the evaluation of $b_L$ at the point of
$BG(k_v)$ represented by $\xi_v$.  Hence, by the definition of Kottwitz's
local pairing,
\[
 \langle\xi_v,L_v\rangle_v
 =\inv_v\bigl(\partial_{L,v}(\xi_v)\bigr)
 =\inv_v\bigl(b_L(\xi_v)\bigr).
\]
Summing over all places gives the asserted formula.
\end{proof}

\begin{corollary}\label{cor:kottwitz-Brauer-localization}
Let $\varnothing\neq S\subseteq\Omega_k$ be finite.  Via
\[
 |BG(k)|=H^1(k,G),\qquad
 |BG(\bA_k^S)|=\prod\nolimits'_{v\notin S}H^1(k_v,G),\qquad
 \Br(BG)/\Br(k)\simeq\Pic(G),
\]
the projection away from $S$ of Kottwitz's localisation sequence gives
\[
 \operatorname{im}\bigl(BG(k)\longrightarrow BG(\bA_k^S)\bigr)
 =BG(\bA_k^S)^{\Br(BG)}.
\]
Thus it agrees with Corollary~\ref{cor:BG-full-Brauer-localization}.
\end{corollary}

\begin{proof}
By the theorem
$\ker(\lambda_G)=BG(\bA_k)^{\Br(BG)}$, and projection away from $S$ gives the
claimed equality.
\end{proof}

\section{Unramified Brauer Groups}

\begin{definition}\label{def:unramified-brauer-stack}
Let $\cZ$ be a smooth finite type algebraic stack over a field $k$.  Let
$b\in\Br(\cZ)$.  If $R$ is a DVR over $k$, with fraction field $K$, and
$z:\Spec K\to\cZ$ is a $K$-point, we say that $b$ is \emph{unramified at}
$(R,K,z)$ if
\[
 z^*b\in\operatorname{im}\bigl(\Br(R)\longrightarrow\Br(K)\bigr).
\]
The class $b$ is \emph{unramified} if it is unramified at every such triple.
The \emph{unramified Brauer group} is
\[
 \Brun(\cZ):=\{b\in\Br(\cZ)\mid b\text{ is unramified}\}.
\]
\end{definition}

\begin{definition}\label{def:eventual-evaluation-criterion}
Let $\cX$ be a smooth finite type algebraic stack over a number field $k$.  We
say that $\cX$ satisfies the \emph{finite evaluation criterion} if, for every
$b\in\Br(\cX)$, we have $b\in\Brun(\cX)$ if and only if
 there is a finite set $\Sigma_b\subseteq\Omega_k$ such that
\[
 \inv_v\bigl(b(x_v)\bigr)=0
\]
for every $v\notin\Sigma_b$ and every $x_v\in\cX(k_v)$.
\end{definition}

\begin{theorem}[Loughran--Santens]\label{thm:LS-evaluation-criterion}
Every smooth finite type Deligne--Mumford stack over a number field satisfies
the finite evaluation criterion.
\end{theorem}

\begin{proof}
This is \cite[Theorem~5.23]{LoughranSantens}. We sketch their argument. If
$b$ is unramified, then one can extend it to a smooth compactification of a
presentation of the stack. The existence of $\Sigma_b$ then follows by
spreading out the compactification and using the vanishing of $\Br(\cO_v)$.

The difficult part is to show that if $b$ is ramified then there is no such $\Sigma_b$.
Suppose that $b$ is ramified.  Then there is a
DVR $R/k$ with fraction field $K$ and a morphism
$z:\Spec K\to\cX$ such that
\[
 z^*b\notin\operatorname{im}\bigl(\Br(R)\to\Br(K)\bigr).
\]
Spreading out this datum gives a connected smooth finite type
$k$-scheme $Y$, an open $V\subseteq Y$, and a morphism
$f:V\to\cX$ such that $f^*b\in\Br(V)$ does not extend to $\Br(Y)$.
Harari's  theorem
\cite[Th\'eor\`eme~2.1.1]{Harari} then produces infinitely many
places $v$ and points $y_v\in V(k_v)$ for which
\[
 \inv_v\bigl((f^*b)(y_v)\bigr)\ne 0.
\]
Thus $b$ has nontrivial evaluation on $f(y_v)\in\cX(k_v)$ at
infinitely many places.
\end{proof}


\subsection{Formal modification}

Let $k$ be a number field and let $\cX$ be a smooth finite type algebraic stack
over $k$, and let $B\subseteq\Br(\cX)$ be a finite subgroup.  Put
\[
 B_{\mathrm{un}}:=B\cap\Brun(\cX),
 \qquad
 \Ann(B_{\mathrm{un}}):=
 \{\chi\in B^D\mid \chi|_{B_{\mathrm{un}}}=0\}.
\]
For $v\in\Omega_k$ and $x_v\in\cX(k_v)$, define
\[
 \ev_{v,x_v}:B\longrightarrow\bQ/\bZ,
 \qquad
 b\longmapsto\inv_v\bigl(b(x_v)\bigr).
\]
 
For smooth varieties, the following statement is implicit in the
 argument used to prove Harari's formal lemma; see
\cite[proof of Corollaire~2.6.1, pp.~233--235]{Harari} and
\cite[proof of Theorem~13.4.3, pp.~341--342, and
Remark~13.4.5, p.~343]{ColliotTheleneSkorobogatovBook}.

\begin{proposition}\label{prop:abstract-evaluation-formal}
Assume that $\cX$ satisfies the finite evaluation criterion and let $B$ be as above.  Then there is a
finite set $\Sigma_B\subseteq\Omega_k$ such that, for every
finite $F\supseteq\Sigma_B$,
$$
 \left\langle
  \ev_{v,x_v}:v\notin F,\ x_v\in\cX(k_v)
 \right\rangle
 =\Ann(B_{\mathrm{un}}).
$$

\end{proposition}

\begin{proof}
Since $B_{\mathrm{un}}$ is finite and $\cX$ satisfies the finite evaluation
criterion, there is a finite set $\Sigma_B\subseteq\Omega_k$ such that
\[
 \inv_v\bigl(b(x_v)\bigr)=0
\]
for every $b\in B_{\mathrm{un}}$, every $v\notin\Sigma_B$, and every
$x_v\in\cX(k_v)$.

Fix a finite set $F\supseteq\Sigma_B$, and let $E_B^F\subseteq B^D$ be the
subgroup generated by the characters $\ev_{v,x_v}$ with $v\notin F$.  The
preceding vanishing gives
\[
 E_B^F\subseteq\Ann(B_{\mathrm{un}}).
\]
For a subgroup $C\subseteq B^D$, write
\[
 C^\perp:=\{b\in B:\chi(b)=0\text{ for every }\chi\in C\}.
\]
Suppose that $E_B^F\subsetneq\Ann(B_{\mathrm{un}})$.
Then by finite duality and
$\Ann(B_{\mathrm{un}})^\perp=B_{\mathrm{un}}$ we would obtain
\[
 (E_B^F)^\perp\supsetneq B_{\mathrm{un}}.
\]
Choose $b\in(E_B^F)^\perp\setminus B_{\mathrm{un}}$.  Then
\[
 \inv_v\bigl(b(x_v)\bigr)=0
\]
for every $v\notin F$ and every $x_v\in\cX(k_v)$.  The finite evaluation
criterion implies that $b\in\Brun(\cX)$, contradicting
$b\notin B_{\mathrm{un}}$.  Hence
$E_B^F=\Ann(B_{\mathrm{un}})$, as required.
\end{proof}

\begin{lemma}\label{lem:finite-character-selection}
Let $A$ be a finite abelian group and let $(S_i)_{i\in I}$ be a family of
subsets of $A$.  Suppose that
\[
 \left\langle\bigcup_{i\in I\setminus F}S_i\right\rangle=A
\]
for every finite subset $F\subseteq I$.  Then, for every $a\in A$ and every
finite $F\subseteq I$, there are pairwise distinct indices
$i_1,\ldots,i_r\in I\setminus F$ and elements $s_j\in S_{i_j}$ such that
\[
 a=s_1+\cdots+s_r,
 \qquad
 r\leq |A|-1.
\]
\end{lemma}

\begin{proof}
Put
\[
 S_\infty:=\{s\in A:s\in S_i\text{ for infinitely many }i\}.
\]
Since $A$ is finite, all elements of $A\setminus S_\infty$ occur among the
$S_i$ for only finitely many indices in total.  The hypothesis, after deleting
those indices, therefore gives $\langle S_\infty\rangle=A$.

Choose a minimal expression
\[
 a=s_1+\cdots+s_r,
 \qquad s_j\in S_\infty,
\]
with nonnegative multiplicities; such an expression exists because, in a
finite group, the submonoid and the subgroup generated by a set coincide.
For $0\leq j\leq r$, put
\[
 p_j:=s_1+\cdots+s_j,
 \qquad p_0:=0.
\]
The partial sums $p_0,\ldots,p_r$ are pairwise distinct.  Indeed, if
$p_i=p_j$ for some $0\leq i<j\leq r$, then
$s_{i+1}+\cdots+s_j=0$, and deleting this block gives a shorter expression
for $a$, contradicting minimality.  Thus $r+1\leq |A|$, and hence
$r\leq |A|-1$.  Finally, each $s_j$ occurs in
infinitely many of the $S_i$, so the corresponding indices $i_j$ may be chosen
pairwise distinct and outside $F$.
\end{proof}

For $x\in\cX(\bA_k)$, let $x_v$ denote its pullback to
$\cX(k_v)$.  We occasionally write $x=(x_v)_v$ abusively even though the product formula
may not hold for $\cX$.
In what follows, we will also use that an adelic point may be modified at finitely many
places without assuming the product formula.  Indeed, let
$I\subseteq\Omega_k$ be finite and choose a finite-presentation integral model
$\scX/\cO_{k,S}$.  By the proof of
Theorem~\ref{thm:stack-product-formula}, after enlarging a finite set
$T\supseteq S\cup I\cup\Omega_k^\infty$, the point $x$ is represented over $\bA_k(T)$ and
\[
 \scX(\bA_k(T))
 \simeq
 \prod_{v\in T}\cX(k_v)\times\scX(\widehatO{T}).
\]
Thus one may replace the factors $x_v$ for $v\in I$ by arbitrary prescribed
points of $\cX(k_v)$ while retaining the remaining factors.

\begin{theorem}[Formal modification]
\label{thm:abstract-formal-modification}
Assume that $\cX$ satisfies the finite evaluation criterion.  For every finite
$\Sigma_0\subseteq\Omega_k$ and every
$x=(x_v)_v\in\cX(\bA_k)^{B_{\mathrm{un}}}$, there is
$x'=(x'_v)_v\in\cX(\bA_k)^B$ such that
\[
 x'_v\simeq x_v\quad(v\in\Sigma_0),
 \qquad
 \#\{v\in\Omega_k:x'_v\not\simeq x_v\}
 \leq |B/B_{\mathrm{un}}|-1.
\]
\end{theorem}

\begin{proof}
Choose $\Sigma_B$ as in Proposition~\ref{prop:abstract-evaluation-formal}.  By
Lemma~\ref{lem:BM-finite-support}, enlarge $\Sigma_0$ to a finite set $\Sigma$
containing $\Sigma_B$ and all archimedean places such that
\[
 \inv_v\bigl(b(x_v)\bigr)=0
\]
for every $b\in B$ and every $v\notin\Sigma$.  The full Brauer--Manin sum
defines a character
\[
 \lambda_x:B\longrightarrow\bQ/\bZ,
 \qquad
 \lambda_x(b):=\sum_{v\in\Sigma}\inv_v\bigl(b(x_v)\bigr).
\]
Since $x$ is orthogonal to $B_{\mathrm{un}}$, we have
$-\lambda_x\in\Ann(B_{\mathrm{un}})$.

For $v\notin\Sigma$, put
\[
 S_v:=\{\ev_{v,z}:z\in\cX(k_v)\}\subseteq\Ann(B_{\mathrm{un}}).
\]
Proposition~\ref{prop:abstract-evaluation-formal} shows that these sets continue to
generate $\Ann(B_{\mathrm{un}})$ after deleting any finite set of places.
Lemma~\ref{lem:finite-character-selection} therefore gives pairwise distinct
places $v_1,\ldots,v_r\notin\Sigma$ and points
$z_i\in\cX(k_{v_i})$ such that
\[
 -\lambda_x=\sum_{i=1}^r\ev_{v_i,z_i},
 \qquad
 r\leq|\Ann(B_{\mathrm{un}})|-1
   =|B/B_{\mathrm{un}}|-1.
\]
Replace $x_{v_i}$ by $z_i$ and leave all other components unchanged. As the
discussion before the statement of the theorem shows, this defines an adelic
point $x'$. The choice of $\Sigma$ and the identity give, for every $b\in B$,
\[
 \sum_v\inv_v\bigl(b(x'_v)\bigr)
 =\lambda_x(b)+\sum_{i=1}^r\inv_{v_i}\bigl(b(z_i)\bigr)=0.
\]
Thus $x'\in\cX(\bA_k)^B$.
\end{proof}


\subsection{Unramified Brauer groups of quotient stacks.}

\begin{proposition}\label{prop:Brauer-detecting-unramified}
Let $\cX$ be a smooth finite type algebraic stack over $k$, and let
$q:Y\to\cX$ be a presentation with $Y$ of finite type over $k$.  Suppose that
\begin{enumerate}
\item for every field extension $K/k$, the map
\[
 Y(K)\longrightarrow |\cX(K)|
\]
is surjective;
\item pullback is an isomorphism
\[
 q^*:\Br(\cX)\xrightarrow{\sim}\Br(Y).
\]
\end{enumerate}
Then
$$
 q^*\Brun(\cX)=\Brun(Y).
$$
Hence $\cX$ satisfies the finite evaluation criterion.
\end{proposition}

\begin{proof}
Functoriality gives
\[
 q^*\Brun(\cX)\subseteq\Brun(Y).
\]
Conversely, let $b\in\Br(\cX)$ and suppose that $q^*b$ is unramified on $Y$.
Let $R$ be a DVR over $k$, with fraction field $K$, and let
$\xi:\Spec K\to\cX$ be a $K$-point.  By (i), there is $y\in Y(K)$ such that
$\xi$ is isomorphic to $q(y)$.  Hence
\[
 \xi^*b=y^*q^*b
 \in\operatorname{im}\bigl(\Br(R)\longrightarrow\Br(K)\bigr).
\]
Thus $b$ is unramified, proving the asserted equality.

For every place $v$ and every $y_v\in Y(k_v)$, one has
\[
 \inv_v\bigl((q^*b)(y_v)\bigr)
 =\inv_v\bigl(b(q(y_v))\bigr).
\]
Since $Y(k_v)\to|\cX(k_v)|$ is surjective, $b$ has trivial
evaluations on all local points of $\cX$ if and only if $q^*b$ does so on
$Y$.  Theorem~\ref{thm:LS-evaluation-criterion}, applied to the smooth finite
type algebraic space $Y$, together with this equality, proves
the finite evaluation criterion for $\cX$.
\end{proof}

\begin{proposition}
  \label{prop:quotient-presentation-unramified}
  Let $G$ be a linear algebraic group acting on a smooth separated geometrically
integral $k$-variety $X$.  Choose a faithful representation
\[
 G\hookrightarrow H:=\SL(V)
\]
and put
\[
 Y:=X\times^G H.
\]
Suppose that $Y$ is a scheme.
Then $Y$ is smooth, separated, and geometrically integral, and the quotient
presentation
\[
 q:Y\longrightarrow [Y/H]\simeq [X/G]
\]
satisfies the hypotheses of
Proposition~\ref{prop:Brauer-detecting-unramified}.
\end{proposition}

\begin{proof}
The fact that
$[X/G]\simeq[Y/H]$ is standard; see \cite[\S2.1]{DhillonClassifying}.  After the fpqc
base change $H\to H/G$, the morphism $Y\to H/G$ becomes the projection
$X\times H\to H$.  It follows that $Y$ is smooth, separated, and
geometrically integral.

For every field extension $K/k$, one has $H^1(K,H)=H^1(K,\SL(V))=1$,
so $Y(K)\to|[Y/H](K)|$ is surjective.  Finally, pullback is an isomorphism
\[
 q^*:\Br([Y/H])\xrightarrow{\sim}\Br(Y)
\]
by \cite[Proposition~5.3]{DhillonClassifying}.
\end{proof}

\begin{theorem}\label{thm:quotient-stack-unramified-residue}
Retain the notation and hypotheses of
Proposition~\ref{prop:quotient-presentation-unramified}.
Let $Y^c$ be a smooth proper compactification of $Y$, with open immersion
$j:Y\hookrightarrow Y^c$, and let $D_1,\ldots,D_r$ be the irreducible
codimension-one components of $Y^c\setminus Y$. Then $j^*$ is injective and
\[
\begin{aligned}
 q^*\Brun([X/G])
 &=\Brun(Y)
  =j^*\Br(Y^c)\\
 &=\ker\!\left(
   \Br(Y)\xrightarrow{\oplus_i\partial_{D_i}}
   \bigoplus_{i=1}^r H^1(k(D_i),\bQ/\bZ)
 \right).
\end{aligned}
\]
Here $\partial_{D_i}$ is the residue at the generic point of $D_i$.
\end{theorem}

\begin{proof}
Propositions~\ref{prop:quotient-presentation-unramified} and
\ref{prop:Brauer-detecting-unramified} give
$q^*\Brun([X/G])=\Brun(Y)$.  By
\cite[Lemma~5.13]{LoughranSantens}, $\Brun(Y)$ is the classical unramified
subgroup of $\Br(k(Y))$.  The result follows from
\cite[Corollaire~6.2]{GrothendieckBrauer}.
\end{proof}

\subsection{The unramified Brauer group of a classifying stack}
\label{subsec:cyclically-unramified}

\begin{definition}\label{def:cyclically-unramified}
Let $M$ be a finitely generated free
abelian group with a continuous action of $\Gal(\bar{k}/k)$.  Choose a finite
Galois extension $K/k$ through which the action factors, and put
$\Gamma:=\Gal(K/k)$.  Define
$$
 \Sha^1_{\mathrm{cyc}}(k,M)
 :=\ker\!\left(
 H^1(\Gamma,M)\longrightarrow
 \prod_{\substack{C\subseteq\Gamma\\ C\ \mathrm{cyclic}}}H^1(C,M)
 \right).
$$
\end{definition}

Recall that a subgroup of $\Gal(\bar{k}/k)$ is called \emph{procyclic} if it is closed
and topologically generated by one element.

The following lemma shows that \cref{def:cyclically-unramified} is intrinsic.

\begin{lemma}\label{lem:cyclic-procyclic-description}
Under the natural inflation isomorphism
$
 H^1(\Gamma,M)\xrightarrow{\sim}H^1(k,M)
$
the subgroup $\Sha^1_{\mathrm{cyc}}(k,M)$ is identified with
$$
 \left\{\alpha\in H^1(k,M)\mid
 \alpha|_P=0\text{ for every procyclic subgroup }
 P\subseteq\Gal(\bar{k}/k)\right\}.
$$
\end{lemma}

\begin{proof}
Put $\Delta:=\Gal(\bar{k}/K)$.  Since $\Delta$ acts trivially on $M$,
$$
 H^1(\Delta,M)=\Hom_{\mathrm{cont}}(\Delta,M)=0:
$$
the image of a continuous homomorphism from the profinite group $\Delta$ to
the discrete group $M$ is finite, hence trivial because $M$ is torsion-free.
Inflation--restriction therefore gives
$H^1(\Gamma,M)\simeq H^1(k,M)$.

Let $\pi:\Gal(\bar{k}/k)\to\Gamma$ be the quotient map.  If a class on
$\Gamma$ vanishes on every cyclic subgroup, then its inflation vanishes on
every procyclic subgroup $P$, since $\pi(P)$ is cyclic.  Conversely, for a
cyclic subgroup $C=\langle\gamma\rangle\subseteq\Gamma$, choose a lift
$\widetilde\gamma\in\Gal(\bar{k}/k)$ and put
$P=\overline{\langle\widetilde\gamma\rangle}$.  Then $P$ is procyclic and
$\pi(P)=C$.  Inflation--restriction for
$1\to P\cap\Delta\to P\to C\to1$ gives an injection
$$
 H^1(C,M)\longrightarrow H^1(P,M),
$$
so vanishing on every such $P$ implies vanishing on every $C$.
\end{proof}

We next compute the unramified Brauer group of $BG$.  Consider the Galois module
\[
 \widehat G:=X^*(G_{\bar{k}}).
\]

There is a natural injection
\[
 \iota_G:H^1(k,\widehat G)\lhook\joinrel\longrightarrow\Pic(G)
\]
whose image is
\[
 \ker\!\left(\Pic(G)\longrightarrow\Pic(G_{\bar{k}})\right).
\]
Indeed, if $G^{\mathrm{tor}}$ denotes the maximal torus quotient of $G$,
then $\widehat{G^{\mathrm{tor}}}=\widehat G$, and $\iota_G$ is the composite
\[
 H^1(k,\widehat G)
 \xrightarrow{\sim}\Pic(G^{\mathrm{tor}})
 \longrightarrow\Pic(G),
\]
where the last map is pullback.  These statements, including their
compatibility with the maps induced by torsors, are proved in
\cite[proof of Proposition~2.9]{ColliotTheleneXu}.  We henceforth regard
$\Sha^1_{\mathrm{cyc}}(k,\widehat G)$ as a subgroup of $\Pic(G)$ through
$\iota_G$.  When $G$ is a torus, $\Pic(G_{\bar{k}})=0$, so $\iota_G$ is an
isomorphism.

\begin{theorem}\label{thm:Brun-BG-Sha}
Let $G$ be a connected linear algebraic group over $k$.  Under the
isomorphism of Proposition~\ref{prop:Pic-Br-BG}, there is a natural
identification
\begin{equation}\label{eq:Brun-BG-quotient}
 \Brun(BG)/\Br(k)
 \xrightarrow{\sim}
 \Sha^1_{\mathrm{cyc}}(k,\widehat G).
\end{equation}
\end{theorem}

\begin{proof}
Choose a faithful representation
\[
 G\hookrightarrow H:=\SL(V)
\]
and put
\[
 Y:=H/G.
\]
Then
\[
 BG\simeq[Y/H].
\]
Let
\[
 q:Y\longrightarrow BG,
 \qquad
 y_0:=1G\in Y(k).
\]
Choose a smooth proper compactification
\[
 j:Y\hookrightarrow Y^c,
\]
which exists in characteristic zero by resolution of singularities
\cite{Hironaka}.
Applying Theorem~\ref{thm:quotient-stack-unramified-residue} with
$X=\Spec k$ gives
$$
 q^*\Brun(BG)=j^*\Br(Y^c).
$$

By Proposition~\ref{prop:Pic-Br-BG}, pullback along $q$ identifies
$\Bre(BG)$ with
\[
 \Br_{y_0}(Y):=
 \ker\bigl(y_0^*:\Br(Y)\longrightarrow\Br(k)\bigr).
\]
Since $j^*$ is injective and evaluation at $y_0$ commutes with
restriction, the preceding equality gives
\[
 q^*\bigl(\Brun(BG)\cap\Bre(BG)\bigr)
 =
 j^*\Br_{y_0}(Y^c),
\]
where
\[
 \Br_{y_0}(Y^c):=
 \ker\bigl(y_0^*:\Br(Y^c)\longrightarrow\Br(k)\bigr).
\]
By \cite[Proposition~2.10(iii)]{ColliotTheleneXu}, there is an isomorphism
\[
 \Br_{y_0}(Y^c)
 \xrightarrow{\sim}
 \Sha^1_{\mathrm{cyc}}(k,\widehat G).
\]
Let
\[
 \theta^c:\Sha^1_{\mathrm{cyc}}(k,\widehat G)
 \xrightarrow{\sim}\Br_{y_0}(Y^c)
\]
denote its inverse.  The $G$-torsor $H\to Y=H/G$ also defines a map
\[
 \theta:H^1(k,\widehat G)\longrightarrow\Br_{y_0}(Y).
\]
Let
\[
 \delta:\Pic(G)\xrightarrow{\sim}\Br_{y_0}(Y)
\]
be the  isomorphism used in the
 proof of Proposition~\ref{prop:Pic-Br-BG}.  The proof of 
\cite[Proposition~2.9]{ColliotTheleneXu}, together with
\cite[Theorem~5.1(iv)]{ColliotTheleneKunyavskii}, shows that the following diagram
\[
\begin{tikzcd}[column sep=large,row sep=large]
 \Sha^1_{\mathrm{cyc}}(k,\widehat G)
   \arrow[r,hook,"\iota_G"]
   \arrow[d,"\theta^c"']
 & \Pic(G) \arrow[d,"\delta"] \\
 \Br_{y_0}(Y^c) \arrow[r,hook,"j^*"']
 & \Br_{y_0}(Y).
\end{tikzcd}
\]
commutes.
More precisely, for every
$\alpha\in\Sha^1_{\mathrm{cyc}}(k,\widehat G)$,
\[
 j^*\bigl(\theta^c(\alpha)\bigr)
 =\theta(\alpha)
 =\delta\bigl(\iota_G(\alpha)\bigr).
\]
The diagram and the preceding computation of the image under $q^*$ give
\[
 \Brun(BG)\cap\Bre(BG)
 \xrightarrow{\sim}
 \Sha^1_{\mathrm{cyc}}(k,\widehat G).
\]

Let $\pi:BG\to\Spec k$ be the structural morphism.  Since
\[
 e^*\circ\pi^*=\operatorname{id}_{\Br(k)},
\]
the point classifying the trivial torsor gives a direct-sum
decomposition
\[
 \Br(BG)=\pi^*\Br(k)\oplus\Bre(BG).
\]
Identifying $\Br(k)$ with its image under $\pi^*$, we
therefore obtain
\[
 \Brun(BG)
 =
 \Br(k)\oplus\bigl(\Brun(BG)\cap\Bre(BG)\bigr).
\]
This proves \eqref{eq:Brun-BG-quotient}.

\end{proof}


\subsection[Approximation with respect to the unramified Brauer group of BG]{Approximation with respect to the unramified Brauer group of \texorpdfstring{$BG$}{BG}}
\label{subsec:Brun-approximation-BG}

Recall from Subsection~\ref{subsec:kottwitz-comparison} the local maps
\[
 \kappa_v:H^1(k_v,G)\longrightarrow\Pic(G)^D.
\]
For each place $v$, put
\[
 E_v:=\operatorname{im}(\kappa_v).
\]
Since $H^1(k_v,G)$ need not be a group, it is a nontrivial fact that $E_v$
is a subgroup. For nonarchimedean $v$, local Kottwitz duality identifies
$E_v$ with the image of the dual restriction map
\[
 r_v^D:\Pic(G_{k_v})^D\longrightarrow\Pic(G)^D
\]
\cite[Theorem~9.1(ii)]{ColliotTheleneFlasque}. At complex places
$E_v=0$, while at real places the assertion follows because evaluation
factors through the abelian cohomology group and the local abelianisation
map is surjective
\cite[Theorem~5.4]{BorovoiAbelian}. The case of a general connected linear
algebraic group follows by quotienting by the unipotent radical.
 
For a finite set
$S\subseteq\Omega_k$, put
\[
 E_S:=\sum_{v\in S}E_v.
\]
Via Proposition~\ref{prop:Pic-Br-BG}, let $U_G\subseteq\Pic(G)$ be the subgroup
corresponding to $\Brun(BG)\cap\Bre(BG)$, and write
\[
 \Ann(U_G):=\{\chi\in\Pic(G)^D:\chi|_{U_G}=0\}.
\]
Finally, define
\[
 \Lambda_{G,S}:BG(\bA_k^S)\longrightarrow\Pic(G)^D,
 \qquad
 \Lambda_{G,S}\bigl((\xi_v)_{v\notin S}\bigr)
 :=\sum_{v\notin S}\kappa_v(\xi_v).
\]
The sum is finite.

\begin{proposition}\label{prop:projected-unramified-set}
For every finite set \(S\subseteq\Omega_k\), one has
$$
 BG(\bA_k^S)^{\Brun(BG)}
 =
 \Lambda_{G,S}^{-1}\bigl(E_S+\Ann(U_G)\bigr).
$$
\end{proposition}

\begin{proof}
Let $\xi^S=(\xi_v)_{v\notin S}\in BG(\bA_k^S)$. By definition,
$\xi^S\in BG(\bA_k^S)^{\Brun(BG)}$ if and only if there exist
$\xi_v\in H^1(k_v,G)$ for $v\in S$ such that the adelic family obtained by
adjoining these components is orthogonal to $\Brun(BG)$.

Under the identification $\Bre(BG)\simeq\Pic(G)$, we may write
$\Brun(BG)=\Br(k)\oplus U_G$. For $L\in U_G$, the proof of
Theorem~\ref{thm:kottwitz-global-sequence} gives
$$
\inv_v(b_L(\xi_v)) = \kappa_v(\xi_v)(L).
$$
Hence
\[
 \xi^S\in BG(\bA_k^S)^{\Brun(BG)}
 \Longleftrightarrow
 \exists e\in E_S:\
 \bigl(\Lambda_{G,S}(\xi^S)+e\bigr)|_{U_G}=0
 \Longleftrightarrow
 \Lambda_{G,S}(\xi^S)\in E_S+\Ann(U_G).
\]
The last equivalence uses $-E_S=E_S$.
\end{proof}

\begin{theorem}\label{thm:S-local-criterion}
Let $S\subseteq\Omega_k$ be finite.  The following are equivalent:
\begin{enumerate}[(i)]
\item
$$
 \Ann(U_G)\subseteq E_S;
$$
\item the localisation map
$$
 BG(k)\longrightarrow BG(\bA_k^S)^{\Brun(BG)}
$$
is surjective.
\end{enumerate}
\end{theorem}

\begin{proof}
Exactness of \eqref{eq:kottwitz-global-sequence} gives
\begin{equation}\label{eq:global-image-ES}
 \operatorname{im}\bigl(BG(k)\longrightarrow BG(\bA_k^S)\bigr)
 =\Lambda_{G,S}^{-1}(E_S).
\end{equation}

The forward implication follows directly from
Proposition~\ref{prop:projected-unramified-set}.

Conversely, assume (ii) and let $\chi\in\Ann(U_G)$.
Propositions~\ref{prop:quotient-presentation-unramified} and
\ref{prop:Brauer-detecting-unramified}, applied with $X=\Spec k$, show that
$BG$ satisfies the finite evaluation criterion.  Apply
Proposition~\ref{prop:abstract-evaluation-formal} and
Lemma~\ref{lem:finite-character-selection} to
$$
 B=\Bre(BG)\simeq\Pic(G),
 \qquad
 B_{\mathrm{un}}=\Brun(BG)\cap\Bre(BG)\simeq U_G,
$$
with the exceptional set chosen to contain $S$.  We obtain pairwise distinct
places $v_1,\ldots,v_r\notin S$ and local torsors
$\eta_i\in H^1(k_{v_i},G)$ such that
$$
 \chi=\sum_{i=1}^r\kappa_{v_i}(\eta_i).
$$
Taking the trivial torsor at every other place outside $S$ gives
$\eta^S\in BG(\bA_k^S)$ with
$\Lambda_{G,S}(\eta^S)=\chi$.  By
Proposition~\ref{prop:projected-unramified-set}, the point $\eta^S$ belongs
to the projected unramified Brauer--Manin set. By (ii), it is the localisation
of a global torsor. Equation~\eqref{eq:global-image-ES} therefore gives
$\chi\in E_S$. Thus $\Ann(U_G)\subseteq E_S$.
\end{proof}

\section{Examples}\label{sec:examples}

\subsection[Approximation for BPGLn]{Approximation for
\texorpdfstring{$B\PGL_n$}{BPGLn}}

Let $n\geq2$ and let $S\subseteq\Omega_k$ be finite.  Define
\begin{equation}\label{eq:def-ISn}
 I_S(n):=
 \sum_{v\in S}\inv_v\bigl(\Br(k_v)[n]\bigr)
 \subseteq \frac1n\bZ/\bZ,
\end{equation}
where $\Br(k_v)[n]$ denotes the $n$-torsion in the local Brauer group.

\begin{theorem}
\label{thm:PGLn-approximation-criterion}
Let
\[
 (A_v)_{v\notin S}\in B\PGL_n(\bA_k^S)
\]
be an adelic family of degree-$n$ central simple algebras.  It is the
localisation away from $S$ of a degree-$n$ central simple algebra over $k$ if
and only if
\begin{equation}\label{eq:PGLn-globalization-condition}
 \sum_{v\notin S}\inv_v(A_v)\in I_S(n).
\end{equation}
\end{theorem}

\begin{proof}
For finite places, the invariant map gives
\[
 H^1(k_v,\PGL_n)\simeq\Br(k_v)[n]
 \simeq\frac1n\bZ/\bZ
\]
The group $H^1(k_v,\PGL_n)$ vanishes at complex places; at a real place, the
image of the invariant map is zero for odd $n$, while for even $n$ it is
$\{0,\tfrac12\}$. See
\cite[Chapters~4 and~6]{GilleSzamuely}.

Suppose first that a global degree-$n$ algebra $A$ has localisation $A_v$ for
all $v\notin S$.  The exact sequence
\begin{equation}\label{eq:global-Brauer-exact-sequence}
 0\longrightarrow\Br(k)
 \longrightarrow\bigoplus_{v\in\Omega_k}\Br(k_v)
 \xrightarrow{\ \sum_v\inv_v\ }\bQ/\bZ
 \longrightarrow0
\end{equation}
gives
\[
 \sum_{v\notin S}\inv_v(A_v)
 =-\sum_{v\in S}\inv_v(A\otimes_k k_v)\in I_S(n),
\]
so \eqref{eq:PGLn-globalization-condition} is necessary.

Conversely, assume \eqref{eq:PGLn-globalization-condition}.  Complete the
given family by classes in $\Br(k_v)[n]$ at $v\in S$ so that the total
invariant is zero.  Exactness of
\eqref{eq:global-Brauer-exact-sequence} provides the converse.
\end{proof}

\begin{corollary}
  In the above situation, the following are equivalent:
  \begin{enumerate}
    \item $B\PGL_n(k)\longrightarrow B\PGL_n(\bA_k^S)$ is surjective;
    \item classical strong approximation holds for $B\PGL_n$ off $S$;
    \item $I_S(n)=\frac1n\bZ/\bZ$;
    \item approximation with respect to $\Brun(B\PGL_n)$ holds for $B\PGL_n$ off $S$.
  \end{enumerate}
\end{corollary}

\begin{proof}
By Proposition~\ref{prop:quotient-stacks-liftable} and
Theorem~\ref{thm:BG-adelic-discrete}, the stack $B\PGL_n$ is adelically
liftable and its adelic topology is discrete, so (1) and (2) are equivalent.
Theorem~\ref{thm:PGLn-approximation-criterion} gives the equivalence with
(3), since every element of $\frac1n\bZ/\bZ$ is realised by an adelic
family supported at a single finite place outside $S$.
Finally, $X^*(\PGL_{n,\bar{k}})=0$, so
Theorem~\ref{thm:Brun-BG-Sha} gives $\Brun(B\PGL_n)=\Br(k)$.
\end{proof}

\begin{corollary}
\label{cor:PGLn-localization}
Let $G=\PGL_n$ over $k$, and let $S$ contain a finite place.  Then
\[
 BG(k)\longrightarrow BG(\bA_k^S)
\]
is surjective.  In particular, classical and unramified strong approximation off $S$ hold for
$B\PGL_n$.
\end{corollary}

\begin{proof}
  This is straightforward.
\end{proof}

\begin{corollary}
Let $S\subseteq\Omega_k^\infty$ and let $n\ge3$.
Classical strong approximation fails for $B\PGL_n$ over $k$ off $S$.
\end{corollary}

\begin{proof}
At archimedean places the invariant lies in $\{0,\tfrac12\}$,
so $\frac1n\notin I_S(n)$ for $n\ge3$.
The assertion follows from the above.
\end{proof}


\subsection{Cyclic norm one tori}
\label{subsec:norm-one-torus}

Let $K/k$ be a cyclic extension of degree $m$, write
\[
 \Gamma:=\Gal(K/k)\simeq \bZ/m\bZ,
 \qquad
 T:=\Res^1_{K/k}\Gm,
\]
and, for every place $v$ of $k$, choose a place $w$ of $K$ above $v$.  Let
\[
 D_w\subseteq\Gamma,
 \qquad
 d_v:=|D_w|=[K_w:k_v]
\]
be the decomposition subgroup and its order.  Since $\Gamma$ is abelian,
$d_v$ is independent of the choice of $w$.  This convention includes the
archimedean places: $d_v=2$ when a real place becomes complex, and $d_v=1$ at
a split real or complex place.
Note that
\[
 \Sha^1_{\mathrm{cyc}}(k,\widehat T)=0
\]
and hence Theorem~\ref{thm:Brun-BG-Sha} gives
$$
 \Brun(BT)=\Br(k).
$$
It follows that approximation with respect to the unramified Brauer group is the same as
classical strong approximation. Further, as the topology on $BT(\bA_k^S)$ is discrete,
this is equivalent to surjectivity of the localisation map.

\begin{theorem}
\label{thm:cyclic-norm-one-criterion}
For every finite set $S\subseteq\Omega_k$,
\begin{equation}\label{eq:norm-one-surjectivity-criterion}
 BT(k)\longrightarrow BT(\bA_k^S)
 \quad\text{is surjective}
 \quad\Longleftrightarrow\quad
 \lcm_{v\in S}(d_v)=m.
\end{equation}
Here the least common multiple of the empty family is understood to be $1$.
Equivalently, the decomposition subgroups at the places in $S$ generate
$\Gamma$.
\end{theorem}

\begin{proof}
The character lattice of $T$ fits into the exact sequence of
$\Gamma$-lattices
$$
 0\longrightarrow\bZ
 \longrightarrow\bZ[\Gamma]
 \longrightarrow\widehat T
 \longrightarrow0,
$$
where $1$ maps to the norm element $\sum_{\gamma\in\Gamma}\gamma$.  Since
$\bZ[\Gamma]$ is a free $\bZ[\Gamma]$-module, its positive-degree group
cohomology vanishes, and the standard descent sequence for a torus gives
$$
 \Pic(T)\simeq H^1(k,\widehat T)
 \simeq H^1(\Gamma,\widehat T)
 \simeq H^2(\Gamma,\bZ)
 \simeq\bZ/m\bZ,
$$
see \cite[Lemma~6.9(ii)]{Sansuc} and
\cite[Chapter~I, \S2]{SerreGaloisCohomology}.

For every place $v$, the norm sequence and Hilbert's theorem~90 give
$$
 H^1(k_v,T)
 \simeq
 k_v^\times/N\bigl((K\otimes_k k_v)^\times\bigr)
 \simeq
 k_v^\times/N_{K_w/k_v}(K_w^\times).
$$
Local reciprocity identifies this quotient with $D_w$; in particular it is
cyclic of order $d_v$, see
\cite[Chapter~V, \S1, Theorem~1.3]{Neukirch}.  The global norm-residue map gives an exact
sequence
\begin{equation}\label{eq:cyclic-norm-reciprocity-sequence}
 k^\times/N_{K/k}(K^\times)
 \longrightarrow
 \bigoplus_{v\in\Omega_k}
 k_v^\times/N_{K_w/k_v}(K_w^\times)
 \xrightarrow{\ \sum_v\operatorname{rec}_v\ }
 \Gamma
 \longrightarrow0,
\end{equation}
where each local quotient maps isomorphically onto the subgroup
$D_w\subseteq\Gamma$, see
\cite[Chapter~VI, \S5, 5.5--5.7]{Neukirch}.

An adelic family away from $S$ gives a term in the middle term of \eqref{eq:cyclic-norm-reciprocity-sequence}.
Its image under the last map in
the sequence can be cancelled by choosing local
classes at the places in $S$ exactly when it belongs to the subgroup of
$\Gamma$ generated by the $D_w$ for $v\in S$.  The forward implication
follows from standard facts about cyclic groups.

Conversely, suppose that the least common multiple is smaller than $m$.
Chebotarev's density theorem gives a finite unramified place $u\notin S$ whose
Frobenius generates $\Gamma$, so $D_u=\Gamma$.  Choose at $u$ a local class
whose reciprocity image does not belong to the subgroup generated by the
$D_w$ for $v\in S$, and take the trivial class at every other place outside
$S$.  No choice of local classes at the places in $S$ can make the total
reciprocity image zero.  The resulting point of $BT(\bA_k^S)$ is therefore not
global.  This proves \eqref{eq:norm-one-surjectivity-criterion}.
\end{proof}


\subsection{Biquadratic tori}

\label{subsec:biquadratic-Sha-example}

The preceding examples have $U_G=0$, so orthogonality to the unramified
Brauer group imposes no condition.  We now give an example in which
unramified Brauer approximation holds although classical strong
approximation fails.

Let
$$
 L=\bQ(\sqrt2,\sqrt3),
 \qquad
 \Gamma=\Gal(L/\bQ)\simeq C_2\times C_2,
$$
and consider the three-dimensional torus
$$
 T:=\Res_{L/\bQ}\Gm/\Gm,
$$
where $\Gm$ is embedded diagonally.

\begin{proposition}\label{prop:biquadratic-Pic-Brauer}
There are natural identifications
$$
 \Pic(T)\simeq\bZ/4\bZ,
 \qquad
 U_T=\Sha^1_{\mathrm{cyc}}(\bQ,\widehat T)
      \simeq 2\bZ/4\bZ\simeq\bZ/2\bZ.
$$
Consequently,
$$
 \Br(\bQ)
 \subsetneq \Brun(BT)
 \subsetneq \Br(BT).
$$
\end{proposition}
 
\begin{proof}
After base change to $\overline{\bQ}$, applying the contravariant character
functor to the defining sequence of $T$ gives the top row of the commutative
diagram
\[
\begin{tikzcd}[column sep=small]
 0 \arrow[r]
 & X^*(T_{\overline{\bQ}}) \arrow[r] \arrow[d,"\sim"']
 & X^*\!\left((\Res_{L/\bQ}\Gm)_{\overline{\bQ}}\right)
     \arrow[r,"\mathrm{res}"] \arrow[d,"\sim"']
 & X^*((\Gm)_{\overline{\bQ}}) \arrow[r] \arrow[d,"\sim"']
 & 0 \\
 0 \arrow[r]
 & \widehat T \arrow[r]
 & \bZ[\Gamma] \arrow[r,"\varepsilon"]
 & \bZ \arrow[r]
 & 0.
\end{tikzcd}
\]
Here the top horizontal map labelled $\mathrm{res}$ is restriction along the
diagonal embedding.  Each coordinate character restricts to the identity
character of $\Gm$, so under the displayed identifications this map is the
augmentation $\varepsilon$.  Thus the character lattice of $T$ is the
augmentation ideal $I_\Gamma=\ker(\varepsilon)$.
For every subgroup $H\subseteq\Gamma$, the module $\bZ[\Gamma]$ is free
over $\bZ[H]$, and
$\varepsilon(\bZ[\Gamma]^H)=|H|\bZ$.  Taking cohomology therefore gives
$$
 H^1(H,\widehat T)\simeq\bZ/|H|\bZ,
$$
with restriction maps given by reduction modulo the subgroup orders.
Inflation and \cite[Lemma~6.9(ii)]{Sansuc} identify $\Pic(T)$ with
$H^1(\Gamma,\widehat T)\simeq\bZ/4\bZ$.  Every nontrivial cyclic
subgroup of $\Gamma$ has order $2$, so
Definition~\ref{def:cyclically-unramified} gives $U_T=2\bZ/4\bZ$.
The Brauer-group assertions follow from
Proposition~\ref{prop:Pic-Br-BG} and Theorem~\ref{thm:Brun-BG-Sha}.
\end{proof}

\begin{proposition}\label{prop:biquadratic-local-restriction}
Let $v$ be a finite place, choose a place of $L$ above $v$, and let
$D_v\subseteq\Gamma$ be its decomposition group.  Under the identification
$\Pic(T)\simeq\bZ/4\bZ$, the local restriction map
$$
 r_v:\Pic(T)\longrightarrow\Pic(T_{\bQ_v})
$$
has the following form:
$$
\begin{array}{c|c|c|c}
 |D_v| & \Pic(T_{\bQ_v}) & \ker(r_v) & E_v\\ \hline
 1 & 0 & \bZ/4\bZ & 0\\
 2 & \bZ/2\bZ & 2\bZ/4\bZ & \Ann(U_T)\\
 4 & \bZ/4\bZ & 0 & \Pic(T)^D
\end{array}
$$
In the middle row, $r_v$ is reduction modulo $2$.
Recall that $E_v$ and $r_v$ were defined in
Subsections~\ref{subsec:Brun-approximation-BG} and
\ref{subsec:kottwitz-comparison}, respectively.
\end{proposition}

\begin{proof}
Inflation and \cite[Lemma~6.9(ii)]{Sansuc} identify
$$
 \Pic(T_{\bQ_v})
 \simeq H^1(D_v,\widehat T),
$$
with $r_v$ induced by restriction from $\Gamma$ to $D_v$.  The calculation
in the proof of Proposition~\ref{prop:biquadratic-Pic-Brauer} gives the
first three columns.  By \cite[Theorem~9.1(ii)]{ColliotTheleneFlasque},
$$
 E_v=\operatorname{im}(r_v^D)=\Ann(\ker r_v),
$$
which gives the last column.
\end{proof}

\begin{theorem}\label{thm:biquadratic-Sha-example}
Let $S=\{\infty,5\}$.  The localisation map
$$
 BT(\bQ)\longrightarrow BT(\bA_{\bQ}^S)^{\Brun(BT)}
$$
is surjective, but classical strong approximation fails for $BT$ off $S$.
\end{theorem}

\begin{proof}
Since $L$ is totally real, $E_\infty=0$.  At $5$, both $2$ and $3$ are
non-squares modulo $5$, so the Frobenius has order $2$ and $|D_5|=2$.
Proposition~\ref{prop:biquadratic-local-restriction} therefore gives
$$
 E_S=E_5=\Ann(U_T).
$$
Theorem~\ref{thm:S-local-criterion} proves the asserted surjectivity.

The extensions $\bQ_3(\sqrt2)/\bQ_3$ and
$\bQ_3(\sqrt3)/\bQ_3$ are respectively unramified and ramified quadratic
extensions.  Their compositum has degree $4$, so $D_3=\Gamma$ and
Proposition~\ref{prop:biquadratic-local-restriction} gives
$E_3=\Pic(T)^D$.  Choose a generator $\psi\in\Pic(T)^D$ and a
torsor $\eta_3\in H^1(\bQ_3,T)$ with $\kappa_3(\eta_3)=\psi$.  Taking
the trivial torsor at every other place outside $S$ gives
$\eta^S\in BT(\bA_{\bQ}^S)$ with $\Lambda_{T,S}(\eta^S)=\psi$.
Since $E_S$ has order $2$, it does not contain $\psi$, so exactness of
\eqref{eq:kottwitz-global-sequence} shows that $\eta^S$ is not global.
The topology on $BT(\bA_{\bQ}^S)$ is discrete by
Theorem~\ref{thm:BG-adelic-discrete}, so classical strong approximation
fails off $S$.
\end{proof}


\section{Classifying stacks of abelian varieties}\label{sec:abelian-varieties}

Let $A/k$ be an abelian variety over a number field, and let $A^\vee$ be
its dual. We will show that the stack $BA$ is adelically liftable, but if $A\neq0$, it has
no topological adelic lifting presentation.  Its adelic points nevertheless
carry the product topology defined using the product formula of
Theorem~\ref{thm:stack-product-formula}.
We prove that this topology is discrete and that, when $A^\vee(k)$ has
positive rank, strong approximation with respect to the full cohomological
Brauer group fails off every finite set of places.  The proof given here does not
need a finiteness
assumption on the Tate--Shafarevich group.


\subsection{The product topology}\label{subsec:BA-topology}

\begin{proposition}\label{prop:relative-classifying-liftable}
Let $\scY/\cO_{k,S}$ be adelically liftable, with adelic lifting presentation
$P\to \scY$.  Let
$$
 \scA\longrightarrow \scY
$$
be a smooth separated group algebraic space of finite presentation with
geometrically connected fibres; in particular, $\scA$ may be an abelian scheme
over $\scY$.  Then the relative classifying stack $B_{\scY}\scA$ is
adelically liftable.
\end{proposition}

\begin{proof}
The morphism $\scY\to B_{\scY}\scA$ classifying the trivial $\scA$-torsor is
smooth, representable, and surjective.  Thus
$$
 P\longrightarrow \scY\longrightarrow B_{\scY}\scA
$$
is a smooth presentation by a separated algebraic space.  The relative
diagonal of $B_{\scY}\scA\to\scY$ is representable by separated algebraic
spaces of finite presentation, since its isomorphism sheaves are fppf
locally isomorphic to $\scA$.  Let $f:B_{\scY}\scA\to\scY$ be the structure
morphism.  The absolute diagonal factors as
\[ 
\begin{tikzcd}[column sep=large, row sep=large]
 B_{\scY}\scA
   \arrow[r,"\Delta_{B_{\scY}\scA/\scY}"']
 & B_{\scY}\scA\times_{\scY}B_{\scY}\scA
   \arrow[r,"j"]
   \arrow[d]
 & B_{\scY}\scA\times_{\cO_{k,S}}B_{\scY}\scA
   \arrow[d,"f\times f"] \\
 & \scY
   \arrow[r,"\Delta_{\scY/\cO_{k,S}}"']
 & \scY\times_{\cO_{k,S}}\scY.
\end{tikzcd}
\]
The square on the right is $2$-Cartesian, so $j$ is a base change of
$\Delta_{\scY/\cO_{k,S}}$.  Since $\scY$ is adelically liftable, this
diagonal, and hence $j$, is representable by separated algebraic spaces of
finite presentation.  These properties are stable under composition, so
they also hold for the absolute diagonal
$\Delta_{B_{\scY}\scA/\cO_{k,S}}=j\circ\Delta_{B_{\scY}\scA/\scY}$.

Let $v\notin S$ and let $\xi\in B_{\scY}\scA(\cO_v)$.  It consists of an
object $y\in\scY(\cO_v)$ and an $\scA_y$-torsor $Q$.
By the lifting property of $P\to\scY$, we may assume that $y$ comes from
$P(\cO_v)$.  Lang's theorem \cite{Lang} gives a point on the special fibre
of $Q$, because the special fibre of $\scA_y$ is connected.
Smoothness and Hensel's lemma lift this point to $Q(\cO_v)$.
Thus $Q$ is trivial and $\xi$ lifts to $P(\cO_v)$.
\end{proof}

In particular, after extending $A$ to an abelian scheme over a suitable
$\cO_{k,S}$, the presentation $\Spec\cO_{k,S}\to B\scA$ proves that $BA$
is adelically liftable.  The next two results explain why a topological
adelic lifting presentation cannot exist when $A\neq0$.

\begin{lemma}\label{lem:bounded-torsor-specialisations}
Let $F$ be a field, let $A/F$ be an abelian variety, and let $U$ be an
algebraic space of finite type over $F$.  For every class
$$
 \alpha\in H^1_{\mathrm{fppf}}(U,A_U)
$$
there is an integer $N>0$ such that
$$
 N\alpha(u)=0\qquad\text{in }H^1(L,A)
$$
for every field extension $L/F$ and every $u\in U(L)$.
\end{lemma}

\begin{proof}
Field-valued points factor through $U_{\mathrm{red}}$, so we may assume
that $U$ is reduced.  We argue by Noetherian induction.  Choose a dense
open subscheme of $U$ whose irreducible components are disjoint, and let
$\eta_1,\ldots,\eta_r$ be its generic points.  Each $\alpha_{\eta_i}$ is
torsion: its torsor has a point over a finite separable extension of
$\kappa(\eta_i)$, and restriction followed by corestriction multiplies
the class by the degree of that extension.

Choose $n>0$ killing all these generic classes.  A section of the torsor
representing $n\alpha$ over each $\eta_i$ spreads out, by finite
presentation, to a nonempty open neighbourhood.  Their union is a
dense open $U_0\subseteq U$ on which $n\alpha=0$.  By induction, there is
an integer $m>0$ killing every field-valued specialisation of
$\alpha$ on the reduced closed complement $U\setminus U_0$.
Then $N=mn$ kills the specialisation at every field-valued point of $U$.
\end{proof}

\begin{proposition}\label{prop:BA-not-topologically-liftable}
Let $A/k$ be a nonzero abelian variety.  No morphism
$$
 q:U\longrightarrow BA
$$
from a finite type algebraic space is surjective on isomorphism classes
of $k_v$-points for any finite place $v$ of $k$.  In particular, $BA$
admits no topological adelic lifting presentation.
\end{proposition}

\begin{proof}
The morphism $q$ corresponds to an $A$-torsor over $U$.
Lemma~\ref{lem:bounded-torsor-specialisations} gives an integer $N>0$
such that
$$
 q\bigl(U(L)\bigr)\subseteq H^1(L,A)[N]
$$
for every field extension $L/k$.
 
Fix a finite place $v$, put $K=k_v$, and let $p$ be its residue
characteristic. Local Tate duality gives
$$
 H^1(K,A)\simeq
 \Hom_{\mathrm{cont}}\bigl(A^\vee(K),\bQ/\bZ\bigr);
$$
see \cite[Chapter~I, Theorem~3.2]{MilneADT}.
By Mattuck's theorem \cite[Theorem 7]{Mattuck}, the compact group $A^\vee(K)$
contains an open subgroup isomorphic to
$\bZ_p^{[K:\bQ_p]\dim A}$.  Pontryagin duality therefore gives a surjection
$$
 H^1(K,A)\longrightarrow
 (\bQ_p/\bZ_p)^{[K:\bQ_p]\dim A}.
$$
The target has unbounded exponent, so
$H^1(K,A)[N]\subsetneq H^1(K,A)$.  Hence $q$ is not surjective on
$k_v$-points.
\end{proof}

We use the local stack topology of \cite[Section~2.4]{Cesnavicius}
and the product topology on adelic points, as defined in Section~2.1.
This product topology is obtained by taking the filtered colimit of the
products at finite adelic levels.

\begin{lemma}\label{lem:topology}
For every place $v$, the space $|BA(k_v)|=H^1(k_v,A)$ is discrete.
\end{lemma}

\begin{proof}
Put $F=k_v$ and fix an $A_F$-torsor $T$.
For a finite type $F$-scheme $U$ with a morphism $U\to BA$, let
$\mathcal P\to U$ be the corresponding torsor.
The sheaf
$$
 \Isom_{A_U}(T\times_F U,\mathcal P)\longrightarrow U
$$
is a smooth $A_U$-torsor.  Its image on $F$-points is the inverse image of the
singleton $\{[T]\}\subseteq H^1(F,A)$.
Smooth morphisms induce open maps on local-field points
\cite[Proposition~2.9(a)]{Cesnavicius}.
Thus this inverse image is open for every $U$, which is precisely the
criterion for $\{[T]\}$ to be open in the local stack topology.
\end{proof}

\begin{proposition}\label{prop:BA-adelic-discrete}
For every finite set $S\subseteq\Omega_k$, the product formula gives
natural identifications
$$
 |BA(\bA_k)|=\prod_{v\in\Omega_k}'H^1(k_v,A),
 \qquad
 |BA(\bA_k^S)|=\prod_{v\notin S}'H^1(k_v,A).
$$
Both spaces are discrete in their product topologies.
\end{proposition}

\begin{proof}
Choose a finite set $F$ containing the archimedean places such that $A$
extends to an abelian scheme $\scA/\cO_{k,F}$.
For $v\notin F$, Lang's theorem and Hensel's lemma give
$$
 H^1(\cO_v,\scA)=0.
$$
The result now follows from the same argument as in
Theorem~\ref{thm:BG-adelic-discrete}.
\end{proof}


\subsection{The Brauer group}\label{subsec:BA-Brauer}

Let $\pi:BA\to\Spec k$ be the structure morphism and let
$s:\Spec k\to BA$ classify the trivial torsor.
Recall that $H^i_{\mathrm{\acute et}}(BA,-)$ denotes cohomology on the big
\'etale site and that
$$
 \Br(BA):=H^2_{\mathrm{\acute et}}(BA,\Gm)_{\mathrm{tors}}.
$$

\begin{proposition}\label{prop:brauer}
There is a natural split decomposition
\begin{equation}\label{eq:h2}
 H^2_{\mathrm{\acute et}}(BA,\Gm)\simeq\Br(k)\oplus A^\vee(k).
\end{equation}
The projection to $\Br(k)$ is $s^*$ and its section is $\pi^*$.
\end{proposition}

\begin{proof}
We apply  cohomological descent for a smooth presentation to
the abelian sheaf $\Gm$. This
gives a first-quadrant spectral sequence
\cite[Proposition~96.20.1, Tag~06XJ]{StacksProject} of the form
$$
 E_1^{p,q}=H^q_{\mathrm{\acute et}}(A^p,\Gm)
 \Longrightarrow H^{p+q}_{\mathrm{\acute et}}(BA,\Gm).
$$
Properness and geometric connectedness give
$H^0_{\mathrm{\acute et}}(A^p,\Gm)=k^\times$.
Every face map is the identity on constants, so the alternating-face
complex in row $q=0$ has  differentials that alternate between $0$ and the identity.
Hence $E_2^{p,0}=0$ for $p>0$.

The term in bidegree $(1,1)$ is
$$
 E_2^{1,1}
 =\ker\!\left(\Pic(A)\xrightarrow{\,m^*-p_1^*-p_2^*\,}
                         \Pic(A\times A)\right)
 =\Pic^0(A)=A^\vee(k).
$$
After base change to $\bar k$ and restriction to $\{a\}\times A$, this
gives $t_a^*L\simeq L$ for every $a\in A(\bar{k})$, which characterises algebraic
triviality.  Conversely, the theorem of the square gives the primitive
identity for every algebraically trivial line bundle.
As $E_2^{3,0}=0$ we have that 
 $E_\infty^{1,1}=A^\vee(k)$ and $E_\infty^{2,0}=0$.
Every class in $E_1^{0,2}=\Br(k)$ survives because
$s^*\pi^*=\operatorname{id}$.  The filtration in total degree two gives
$$
 0\longrightarrow A^\vee(k)\longrightarrow
 H^2_{\mathrm{\acute et}}(BA,\Gm)
   \xrightarrow{s^*}\Br(k)\longrightarrow0.
$$
The section $\pi^*$ proves \eqref{eq:h2}.
For the elliptic-curve case, see also
\cite[Proposition~2.4.4]{ShinThesis}.
\end{proof}

\begin{corollary}\label{cor:BA-brauer}
  In the above situation
  $$
 \Br(BA)=\Br(k)\oplus (A^\vee(k))_{\rm tors}.
$$
\end{corollary}

\begin{proof}
  This is clear.
\end{proof}

For $P\in A^\vee(k)$, write $\beta_P$ for the class in
$H^2_{\mathrm{\acute et}}(BA,\Gm)$ corresponding to $P$ under
\eqref{eq:h2}.  The
 identification
\cite[Chapter~I, Lemma~3.1]{MilneADT} associates to $P$ a commutative
extension 
$$
 1\longrightarrow\Gm\longrightarrow G_P\longrightarrow A\longrightarrow0.
$$
Let $L_P$ be the line bundle corresponding
to the underlying $\Gm$-torsor $G_P\to A$.

The extension induces a morphism
$$
 BG_P\longrightarrow BA,
$$
which is a gerbe banded by $\Gm$, see
\cite[Chapter~IV, proof of Theorem~2.5, pp.~143--144]{MilneEtale} for the definition.
Denote the class of this gerbe by
$\gamma_P\in H^2_{\mathrm{\acute et}}(BA,\Gm)$.

\begin{lemma}\label{lem:BA-gerbe-spectral-class}
Let $F^\bullet$ be the filtration on
$H^2_{\mathrm{\acute et}}(BA,\Gm)$ induced by the descent spectral
sequence of Proposition~\ref{prop:brauer}.  Then
$\gamma_P\in F^1=\ker(s^*)$, and its image in
$$
 \operatorname{gr}_F^1 H^2_{\mathrm{\acute et}}(BA,\Gm)
 =E_\infty^{1,1}=E_2^{1,1}=\Pic^0(A)=A^\vee(k)
$$
is $P$.
\end{lemma}

\begin{proof}
Pulling $BG_P\to BA$ back along $s:\Spec k\to BA$ gives the gerbe of
$G_P$-lifts of the trivial $A$-torsor.  The trivial $G_P$-torsor 
provides a point of this gerbe and hence the gerbe is trivial.
In other words,
$s^*\gamma_P=0$, and $\gamma_P$ lies in $F^1$.

On the double overlap
$$
 \Spec k\times_{BA}\Spec k\simeq A,
$$
a point $a$ of $A$ is an automorphism of the trivial $A$-torsor. 
The $\Gm$-torsor of lifts of the universal
automorphism through $G_P\to A$ is $G_P\to A$.  Thus the
degree-$(1,1)$ descent datum of $\gamma_P$ is the class
$$
 [L_P]\in H^1_{\mathrm{\acute et}}(A,\Gm)=\Pic(A).
$$

As $L_P$ is algebraically trivial this class is in the kernel and corresponds
to $P\in A^\vee(k)$.
\end{proof}
 
Lemma~\ref{lem:BA-gerbe-spectral-class} and the split decomposition
\eqref{eq:h2} identify $\gamma_P$ with $\beta_P$.

Now let $K/k$ be a field extension and let $\xi\in H^1(K,A)$.  Pulling
$BG_P\to BA$ back along the point $\xi:\Spec K\to BA$ gives the gerbe of
$G_P$-lifts of the $A$-torsor $\xi$.  By the definition of the connecting
homomorphism, the class of this gerbe is $\delta_P(\xi)$.  Hence
\begin{equation}\label{eq:BA-evaluation}
 \xi^*\beta_P=\delta_P(\xi)\in\Br(K),
\end{equation}
where
$\delta_P:H^1(K,A)\to H^2_{\mathrm{\acute et}}(K,\Gm)=\Br(K)$
is the connecting homomorphism of the extension.

For $K=k_v$, put
$$
 \langle\xi,P\rangle_v:=\inv_v\bigl(\delta_P(\xi)\bigr). 
$$
This is the local Tate pairing (see
\cite[Chapter~I, Theorem~3.2]{MilneADT}); in particular, it is bilinear.

\begin{lemma}\label{lem:BA-reciprocity}
For every $\eta\in H^1(k,A)$ and every $P\in A^\vee(k)$, one has
$$
 \sum_{v\in\Omega_k}\langle\eta_v,P\rangle_v=0.
$$  
\end{lemma}

\begin{proof}
By \eqref{eq:BA-evaluation}, the local evaluations are the restrictions
of the single class $\delta_P(\eta)\in\Br(k)$.
By 
\cite[Chapter~VI, \S5, Theorem~5.8]{Neukirch}, only finitely many of these
restrictions are nonzero and the sum of their local invariants is zero.
This applies even when $\beta_P$ has infinite order.
\end{proof}


\subsection{Failure of strong approximation}\label{subsec:BA-failure}
In this subsection we will show that strong approximation with respect to the full Brauer group
fails on $BA$.
For an abelian group $M$ and a prime $\ell$, write
$$
 M^{\wedge\ell}:=\varprojlim_n M/\ell^nM,
 \qquad
 M[\ell^\infty]:=\bigcup_n M[\ell^n].
$$

\begin{lemma}\label{lem:BA-local-completions}
Let $K/\bQ_p$ be finite and let $C/K$ be an abelian variety.
If $\ell\neq p$, then $C(K)^{\wedge\ell}$ is a finite $\ell$-group.
The kernel of $C(K)\to C(K)^{\wedge p}$ is finite of order prime to $p$.
In particular, a point of infinite order has nonzero image in
$C(K)^{\wedge p}$.
\end{lemma}

\begin{proof} This is standard. We recall the proof for completeness.
As we have previously observed, the group $C(K)$ contains an open subgroup
$U\simeq\bZ_p^d$, where $d=[K:\bQ_p]\dim C$.
Since $C(K)$ is compact, $C(K)/U$ is finite; in particular, $C(K)$ is
profinite.  For each prime $\ell$, let $C(K)_\ell$ denote its pro-$\ell$
Sylow subgroup.  There is a canonical decomposition
$$C(K)\simeq\prod_\ell C(K)_\ell.$$
Since $U$ is pro-$p$, one has $U\subseteq C(K)_p$ and, for $\ell\ne p$,
$U\cap C(K)_{\ell}=0$. Hence $C(K)_{\ell}$ injects into the finite quotient
$C(K)/U$ and is therefore finite.
Moreover, $C(K)$ is topologically finitely generated because it contains
the topologically finitely generated open subgroup $U$.  By
\cite[Theorem~4.3.5, pp.~129--131]{RibesZalesskii}, there is a
noncanonical topological decomposition
$$
 C(K)\simeq\bZ_p^d\oplus F_K
$$
with $F_K$ finite, from which the assertion follows.
\end{proof}

\begin{lemma}\label{lem:BA-local-duality}
For a finite place $v$ and any prime $\ell$, the local Tate pairing
identifies
$$
 H^1(k_v,A)[\ell^\infty]\simeq
 \Hom_{\mathrm{cont}}\!\left(
      A^\vee(k_v)^{\wedge\ell},\bQ_\ell/\bZ_\ell\right).
$$
In particular, for every nonzero element
$x\in A^\vee(k_v)^{\wedge\ell}$, there exists
$\xi\in H^1(k_v,A)[\ell^\infty]$ such that the continuous character
corresponding to $\xi$ under this identification is nonzero at $x$.
\end{lemma}

\begin{proof}
By \cite[Chapter~I, Lemma~3.1]{MilneADT}, there is a canonical identification
$$
 A^\vee(k_v)\simeq \Ext^1_{k_v}(A,\Gm).
$$
Under this identification, local Tate duality
\cite[Chapter~I, Theorem~3.2]{MilneADT} gives an isomorphism
$$
 A^\vee(k_v)\xrightarrow{\sim}
 \Hom_{\mathrm{cont}}\!\left(H^1(k_v,A),\bQ/\bZ\right).
$$
Here $A^\vee(k_v)$ is compact and $H^1(k_v,A)$ is discrete.  Pontryagin
duality therefore gives the transpose isomorphism
$$
 H^1(k_v,A)\xrightarrow{\sim}
 \Hom_{\mathrm{cont}}\!\left(A^\vee(k_v),\bQ/\bZ\right),
 \qquad
 \xi\longmapsto\bigl(P\longmapsto\langle\xi,P\rangle_v\bigr).
$$

Via the canonical isomorphism
$$
 \bQ_\ell/\bZ_\ell\simeq 
 (\bQ/\bZ)[\ell^\infty]=\bZ[1/\ell]/\bZ,
$$
regard $\bQ_\ell/\bZ_\ell$ as the $\ell$-primary subgroup of
$\bQ/\bZ$.  A continuous character of $A^\vee(k_v)$ has $\ell$-power
order if and only if its image lies in $\bQ_\ell/\bZ_\ell$. If such a character
is killed by $\ell^n$, then it annihilates $\ell^nA^\vee(k_v)$ and factors
through $A^\vee(k_v)/\ell^nA^\vee(k_v)$.  Consequently, restriction along
the completion map induces an isomorphism
$$
 \Hom_{\mathrm{cont}}\!\left(A^\vee(k_v)^{\wedge\ell},
                 \bQ_\ell/\bZ_\ell\right)
 \xrightarrow{\sim}
 \Hom_{\mathrm{cont}}\!\left(A^\vee(k_v),
                 \bQ_\ell/\bZ_\ell\right)
$$
and the first result follows.

Finally, if $x\neq0$ in $A^\vee(k_v)^{\wedge\ell}$, then its image is
nonzero in $A^\vee(k_v)/\ell^nA^\vee(k_v)$ for some $n$.  Characters of
this finite abelian $\ell$-group separate points, so some character with
values in $\bQ_\ell/\bZ_\ell$ is nonzero on $x$.  The above
identification gives the corresponding $\ell$-primary local torsor.
\end{proof}

Recall that $H^1(K,A)$ is torsion for every field extension $K/k$:
a torsor has a point over a finite extension, and
restriction--corestriction kills its class by the degree.
At archimedean places, $H^1(\bC,A)=0$ and $H^1(\bR,A)$ is killed by $2$.

\begin{theorem}\label{thm:BA-full-Brauer-failure}
Suppose that $A^\vee(k)$ has positive rank.
For every finite set $S\subseteq\Omega_k$, one has a strict inclusion
$$
 \operatorname{im}\bigl(BA(k)\longrightarrow BA(\bA_k^S)\bigr)
 \subsetneq BA(\bA_k^S)^{\Br(BA)}
$$
and hence strong approximation off $S$ with respect to the Brauer group fails.
\end{theorem}

\begin{proof}
As in Definition~\ref{def:strong-approximation-off-S}, the set on the
right is the projection of $BA(\bA_k)^{\Br(BA)}$.
We construct a point in this full Brauer--Manin set whose projection
away from $S$ does not come from a $k$-point.

By the Mordell--Weil theorem, the group
$T:=A^\vee(k)_{\mathrm{tors}}$ is finite.
Choose an odd rational prime $\ell$ not dividing $|T|$ and different
from the residue characteristic of every finite place in $S$.
Choose $P\in A^\vee(k)$ of infinite order.
Lemma~\ref{lem:BA-local-completions} shows that
$A^\vee(k_v)^{\wedge\ell}$ is a finite $\ell$-group at each finite
$v\in S$.  Thus there is an integer $e\geq0$ such that
$$
 Q:=\ell^eP
$$
has zero image in every one of these completions.

Choose a place $w\mid\ell$; then $w\notin S$.
The injection $A^\vee(k)\to A^\vee(k_w)$ preserves the infinite order
of $Q$.  By Lemma~\ref{lem:BA-local-completions}, its image in
$A^\vee(k_w)^{\wedge\ell}$ is nonzero.
Lemma~\ref{lem:BA-local-duality} supplies a class
$$
 \xi_w\in H^1(k_w,A)[\ell^\infty],
 \qquad
 \langle\xi_w,Q\rangle_w\neq0.
$$
Let $x=(\xi_v)_v$ have this component at $w$ and the trivial component
at every other place.  Proposition~\ref{prop:BA-adelic-discrete} makes
$x$ an adelic point.

By Corollary~\ref{cor:BA-brauer}, every $b\in\Br(BA)$ is uniquely of the form
$\pi^*a+\beta_t$, where $a\in\Br(k)$ and $t\in T$.
For every place $v$, equation~\eqref{eq:BA-evaluation} gives
$$
 \inv_v\bigl(b(x_v)\bigr)
 =\inv_v(a_{k_v})+\langle\xi_v,t\rangle_v.
$$
The sum of the first terms vanishes by global reciprocity.  Since
$\xi_v=0$ for every $v\neq w$, the second term vanishes at every such
place and equals $\langle\xi_w,t\rangle_w$ at $w$.
Bilinearity shows that this value is killed by an $\ell$-power and by
the order of $t$.  These integers are coprime, so the value is zero.
Hence $x\in BA(\bA_k)^{\Br(BA)}$.

Suppose that a global class $\eta\in H^1(k,A)$ has localisation
$\xi_v$ at every $v\notin S$. Note that $H^1(k,A)$ is torsion so 
such a class has a primary decomposition.
Let $\eta_\ell$ be its $\ell$-primary component. 
Localisation respects primary decomposition, so
$$
 (\eta_\ell)_w=\xi_w,
 \qquad
 (\eta_\ell)_v=0\quad(v\notin S\cup\{w\}).
$$ 
At finite $v\in S$, the choice of $Q$ and
Lemma~\ref{lem:BA-local-duality} give
$\langle(\eta_\ell)_v,Q\rangle_v=0$.
At archimedean places, $(\eta_\ell)_v=0$ because $\ell$ is odd.
It follows that
$$
 \sum_v\langle(\eta_\ell)_v,Q\rangle_v
 =\langle\xi_w,Q\rangle_w\neq0,
$$
contrary to Lemma~\ref{lem:BA-reciprocity}.
\end{proof}

\end{document}